\documentclass[twoside,10pt]{amsart}

\usepackage{amsthm,amsmath,amsfonts,epsfig,amssymb}
\usepackage[T1]{fontenc}
\usepackage{graphicx}
\usepackage{enumerate}
\usepackage{xy}
\usepackage{graphics}
\usepackage[colorlinks=true]{hyperref} 

\xyoption{all}
\entrymodifiers={+!!<0pt,\fontdimen22\textfont2>}

\newtheorem{thm}{Theorem}[section]

\newtheorem{prop} [thm]{Proposition}
\newtheorem{lem} [thm]{Lemma}
\newtheorem{coro}[thm]{Corollary}

\newtheorem*{thm*}{Theorem}
\theoremstyle{definition}
\newtheorem{defin}[thm]{Definition}
\newtheorem{nota}[thm]{Notation}
\theoremstyle{remark}
\newtheorem{rem}[thm]{Remark}
\newtheorem{exem}[thm]{Example}

\def\Z{\mathbb{Z}}
\def\N{\mathbb{N}}
\def\A{\mathbb{A}}
\def\O{\mathcal{O}}
\def\L{\mathcal{L}}

\def\cT{\mathcal{T}}

\def\cP{\mathcal{P}}
\def\tP{\tilde{\mathcal{P}}}

\def\id{\mathrm{id}} 

\def\Fr{\mathrm{Fr}} 

\newcommand{\aone}{\A^1}
\def\gm{\mathbb{G}_m} 
\def\gmpt{\mathbb{G}_{m,1}} 

\newcommand{\bpi}{\boldsymbol{\pi}}
\newcommand{\piaone}{{\bpi}^{\aone}}

\def\Hom{\mathrm{Hom}}
\def\homm#1#2#3{\mathrm{Hom}_{#1}(#2,#3)}
\def\uHom{\underline{\mathrm{Hom}}} 

\def\spec#1{\mathrm{Spec}(#1)}

\def\K{\mathrm{K}}

\def\KM{\mathrm{K}^{\mathrm{M}}} 
\def\KMW{\mathrm{K}^{\mathrm{M\hspace{-.2ex}W}}} 

\def\sKM{\mathbf{K}^{\mathrm{M}}} 
\def\sKMW{\mathbf{K}^{\mathrm{M\hspace{-.2ex}W}}} 

\def\I{\mathrm{I}} 

\newcommand{\tch}[1]{%
  \mathchoice{\widetilde{\mathrm{CH}}^{\raisebox{-.5ex}{$\scriptstyle#1$}}}
             {\widetilde{\mathrm{CH}}^{\raisebox{-.8ex}{$\scriptstyle#1$}}}
             {}
             {}
}
\newcommand{\tchi}[2]{%
  \mathchoice{\widetilde{\mathrm{CH}}_{#2}^{\raisebox{-.5ex}{$\scriptstyle#1$}}}
             {\widetilde{\mathrm{CH}}_{#2}^{\raisebox{-.7ex}{$\scriptstyle#1$}}}
             {}
             {}
}
\def\ch#1#2{\tch{#1}(#2)}
\def\cht#1#2#3{\tch{#1}(#2,#3)}
\def\chs#1#2#3{\tchi{#1}{#2}(#3)}
\def\chst#1#2#3#4{\tchi{#1}{#2}(#3,#4)}

\def\CH{\mathrm{CH}}

\def\graph#1{\tilde{\gamma}_{#1}}

\def\Pic{\mathrm{Pic}}

\def\Or{\mathrm{Q}}

\def\GW{\mathrm{GW}} 
\def\W{\mathrm{W}} 

\def\H{\mathrm{H}}
\def\HMW{\mathrm{H}_{\mathrm{MW}}}
\def\HI{\mathrm{H}_{\mathrm{I}}}

\def\tZ{\tilde\Z}
\def\ItZ{\mathrm{I}\tilde\Z}
\def\tZX#1{\tilde {\mathrm {c}}(#1)}
\def\itZX#1{\mathrm{I}\tilde {\mathrm {c}}(#1)}

\def\Ztr#1{\Z_{\mathrm{tr}}(#1)}

\def\ext{\mathrm{ext}}
\def\res{\mathrm{res}}

\def\bc{\mathrm{bc}}
\def\tr{\mathrm{tr}}

\def\ome#1#2{\omega_{#1/#2}}

\def\sm#1{\mathrm{Sm}_{#1}}
\def\barsm#1{\overline{\mathrm{Sm}}_{#1}}

\def\cor#1{\widetilde{\mathrm{Cor}}_{#1}}
\def\icor#1{\widetilde{\mathrm{ICor}}_{#1}}
\def\ucor#1{\mathrm{Cor}_{#1}}
\def\Adm{\mathcal{A}}

\def\unit{\eta}
\def\counit{\epsilon}
\def\tunit{\tilde\eta}
\def\tcounit{\tilde\epsilon}

\def\Tr{\mathrm{Tr}}

\def\Zar{\mathrm{Zar}}
\def\Nis{\mathrm{Nis}}
\def\Et{\mathrm{Et}}

\def\psh#1{\widetilde{\mathrm{PSh}}_{#1}}
\def\sh#1#2{\widetilde{\mathrm{Sh}}_{#2,#1}}
\def\Sh#1{\mathrm{Sh}_{#1}} 

\def\Ci#1{\mathrm{C}_{#1}}
\def\Cstar{\mathrm{C}_*^{\mathrm{sing}}}

\def\DA{\mathrm{D}_{\A^1}} 
\def\DAeff{\mathrm{D}_{\A^1}^{\text{eff}}} 
\def\DM{\mathrm{DM}} 
\def\tDM{\widetilde{\DM}} 

\begin{document}

\title{The category of finite MW-correspondences}

\author{Baptiste Calm\`es}\email{baptiste.calmes@univ-artois.fr}\address{Laboratoire de math\'ematiques de Lens \\
Facult\'e des sciences Jean Perrin \\
Universit\'e d'Artois\\
Rue Jean Souvraz SP 18\\
62307 Lens Cedex\\
France}
\author{Jean Fasel}\email{jean.fasel@gmail.com}\address{Institut Fourier-UMR 5582 \\ Universit\'e Grenoble-Alpes   \\CS 40700 \\
38058 Grenoble Cedex 9 \\
France}

\thanks{The first author acknowledges the support of the French Agence Nationale de la Recherche (ANR) under reference ANR-12-BL01-0005}

\begin{abstract}
We introduce the category of finite MW-correspondences over a perfect field $k$ with $\mathrm{char}(k)\neq 2$. We then define for any essentially smooth scheme $X$ and integers $p,q\in\Z$ MW-motivic cohomology groups $\H^{p,q}(X,\tZ)$ and begin the study of their relationship with ordinary motivic cohomology groups.
\end{abstract}



\maketitle

\pagenumbering{arabic}


\setcounter{tocdepth}{1}
\tableofcontents

\section*{Introduction}

\medskip

Let $k$ be a perfect field and let $\sm{k}$ be the category of smooth separated schemes of finite type over $k$. One of the central ideas of V. Voevodsky in his construction of motivic cohomology is the definition of the category of finite correspondences $\ucor{k}$ (see for instance \cite{Mazza06}). Roughly speaking, the category $\ucor{k}$ is obtained from $\sm{k}$ by taking the smooth schemes as objects and formally adding transfer morphisms $\tilde f:Y\to X$ for any finite surjective morphism $f:X\to Y$ of schemes. There is an obvious functor $\sm{k}\to \ucor{k}$ and the presheaves (of abelian groups) on $\sm{k}$ endowed with transfer morphisms for finite surjective morphisms become naturally presheaves on $\ucor{k}$, also called presheaves with transfers. Classical Chow groups or Chow groups with coefficients \`a la Rost are examples of such presheaves. Having the category of finite correspondences in hand, it is then relatively easy to define motivic cohomology, which is an algebro-geometric analogue of singular cohomology in topology. The analogy between topology and algebraic geometry hinted at above extends in various directions: algebraic $\K$-theory is an analogue of topological $\K$-theory, and there is a motivic Atiyah-Hirzebruch spectral sequence relating motivic cohomology and algebraic $\K$-theory, as the classical Atiyah-Hirzebruch spectral sequence relates topological $\K$-theory and singular cohomology \cite{Friedlander02}.

However, there are also many examples of interesting (pre-)sheaves without transfers in the above sense. Our main examples here are the Chow-Witt groups \cite{Fasel08a} or the cohomology of the (stable) homotopy sheaves $\piaone_i(X,x)$ of a pointed smooth scheme $(X,x)$, most notably the Milnor-Witt $\K$-theory sheaves $\sKMW_n$ for $n\in\Z$. Such sheaves naturally appear in the Gersten-Grothendieck-Witt spectral sequence computing higher Grothendieck-Witt groups, aka Hermitian $\K$-theory \cite{Fasel09c} or in the unstable classification of vector bundles over smooth affine schemes \cite{Asok12a,Asok12b}, and thus are far from being exotic. 

Although these sheaves don't have transfers for general finite morphisms, they do have transfers for finite surjective morphisms with trivial relative canonical sheaf (depending on a trivialization of the latter), and one can hope to formalize this notion and then follow Voevodsky's construction of the derived category of motives from finite correspondences. In his work on the Friedlander-Milnor conjecture, Morel introduced a notion of generalized transfers in order to deal with this situation \cite{Morel11}. Our approach in this article is a bit different in spirit. We enlarge the category of smooth schemes using finite MW-correspondences. Roughly speaking, we replace the Chow groups (or cycles) in Voevodsky's definition by Chow-Witt groups (or cycles with extra quadratic information) and define in this way the category of finite MW-correspondences $\cor k$. The obvious functor $\sm{k}\to \ucor{k}$ factors through our category; namely there are functors $\sm{k}\to \cor k$ and $\cor k\to \ucor{k}$ whose composite is the classical functor. Given $X,Y$ smooth, the homomorphism $\cor k(X,Y)\to \ucor{k}(X,Y)$ is in general neither injective (by far) nor surjective (yet almost). We call the presheaves on $\cor k$ \emph{presheaves with MW-transfers}. It is easy to see that presheaves with MW-transfers in our sense are also presheaves with generalized transfers in Morel's sense, and we believe that the two notions are the same. A presheaf on $\ucor{k}$ is also a presheaf on $\cor k$, but the examples above are genuine presheaves with generalized transfers, so our notion includes many more examples than the classical one. 

Having $\cor k$ at hand, we define MW-motivic cohomology groups $\HMW^{p,q}(X,\Z)$ for any smooth scheme $X$ and any integers $p,q\in \Z$. The main difference with the classical groups is that the MW-motivic cohomology groups are non trivial for $q<0$, in which range they can be identified with the cohomology of the Gersten-Witt complex defined in \cite{Balmer02}. 

We foresee many applications of our main theorem and more generally of MW-motivic cohomology. For instance, it is expected that the MW-motivic cohomology groups will naturally appear in an Atiyah-Hirzebruch-type spectral sequence computing higher Grothendieck-Witt groups (aka Hermitian $\K$-theory). Moreover, these groups should give a precise idea of the stable homology sheaves $\mathbf{H}_i^{\A^1}(\gm^{\wedge n})$ of smash powers of $\gm$, as we now explain. Let $\DAeff(k)$ be the full subcategory of $\A^1$-local objects in the derived category of Nisnevich sheaves of abelian groups, and let $\DA(k)$ be the category obtained from the latter by formally inverting the Tate object. By construction, there is a functor $\DAeff(k)\to \DM(k)$, where $\DM(k)$ is obtained by inverting the Tate object in the full subcategory of $\A^1$-local objects in the derived category of (bounded below) Nisnevich sheaves with transfers. This functor factorizes through the category $\tDM(k)$ defined analogously using $\cor k$ instead of $\ucor{k}$, and thus $\DM(k)$ is in this sense closer to $\DA(k)$ than $\DM(k)$. Now, the stable homology sheaves are computed in $\DA(k)$, and it possible to define analogous sheaves in $\DM(k)$ which should be quite close to the former.

To conclude this introduction, let us mention some works related to this one. First, the approach of the Friedlander-Milnor conjecture of Morel in \cite{Morel11} is the starting point of our definition of the category $\cor k$. As discussed above, his sheaves with generalized transfers should coincide with ours. Second, there is work in progress by M. Schlichting and S. Markett on an Atiyah-Hirzebruch spectral sequence for higher Grothendieck-WItt groups. We hope to prove that the groups they obtain at page $2$ are indeed MW-motivic cohomology groups. Finally, let us mention a recent preprint of Neshitov \cite{Neshitov14} in which computations similar to ours are done in the category of framed correspondences defined by Garkusha and Panin following ideas of Voevodsky \cite{Garkusha14}. 

\subsection*{Acknowledgments}
We wish to thank Aravind Asok and Fr\'ed\'eric D\'eglise for useful remarks on a preliminary version of this article. We are also grateful to Fabien Morel, Oleg Podkopaev and Antoine Touz\'e for some conversations. Finally, we would like to thank Marco Schlichting for explaining us his joint work with Simon Markett on a Grayson type spectral sequence computing higher Grothendieck-Witt groups and Jean Barge for his idea on how to compute the kernel of the map from MW-motivic cohomology to ordinary motivic cohomology.

\subsection*{Conventions}
The schemes are separated of finite type over some perfect field $k$ with $\mathrm{char}(k)\neq 2$. If $X$ is a smooth connected scheme over $k$, we denote by $\Omega_{X/k}$ the sheaf of differentials of $X$ over $\spec k$ and write $\ome{X}{k}:=\det\Omega_{X/k}$ for its canonical sheaf. In general we define $\ome{X}{k}$ connected component by connected component. We use the same notation if $X$ is the localization of a smooth scheme at any point. If $k$ is clear from the context, we omit it from the notation. If $f:X\to Y$ is a morphism of (localizations of) smooth schemes, we set $\omega_{f}=\ome{X}{k}\otimes f^*\ome{Y}{k}^\vee$. If $X$ is a scheme and $n\in\N$, we denote by $X^{(n)}$ the set of codimension $n$ points in $X$.


\section{Milnor-Witt $\K$-theory}

In this section, we recall first the definition of Milnor-Witt $\K$-theory of a field and its associated sheaf following \cite[\S 3]{Morel08}. We then recall the definition of Chow-Witt groups and spend some time on their functorial properties. Following Morel, we don't make any assumption on the characteristic of the field.

For any field $F$, let $\KMW_*(F)$ be the $\Z$-graded associative (unital) ring freely generated by symbols $[a]$, for each $a\in F^\times$, of degree $1$ and by a symbol $\eta$ in degree $-1$ subject to the relations
\begin{enumerate}[(i)]
\item\label{relation:steinberg} $[a][1-a]=0$ for any $a\neq 0,1$.
\item\label{relation:linearity} $[ab]=[a]+[b]+\eta[a][b]$ for any $a,b\in F^\times$.
\item\label{relation:commute} $\eta[a]=[a]\eta$ for any $a\in F^\times$.
\item\label{relation:hyperbolic} $\eta(2+\eta[-1])=0$.
\end{enumerate}

If $a_1,\ldots,a_n\in F^\times$, we denote by $[a_1,\ldots,a_n]$ the product $[a_1]\cdot\ldots\cdot[a_n]$. Let $\GW(F)$ be the Grothendieck-Witt ring of non degenerate bilinear symmetric forms on $F$. Associating to a form its rank yields a surjective ring homomorphism 
\[
\mathrm{rank}:\GW(F)\to \Z
\]
whose kernel is the \emph{fundamental ideal} $\I(F)$. We can consider for any $n\in\N$ the powers $\I^n(F)$ and we set $\I^n(F)=\W(F)$ for $n\leq 0$, where the latter is the Witt ring of $F$. It follows from \cite[Lemma 3.10]{Morel08} that we have a ring isomorphism
\[
\GW(F)\to \KMW_0(F)
\]
defined by $\langle a\rangle\mapsto 1+\eta[a]$. We will thus identify $\KMW_0(F)$ with $\GW(F)$ later on. In particular, we will denote by $\langle a\rangle$ the element $1+\eta[a]$ and by $\langle a_1,\ldots,a_n\rangle$ the element $\langle a_1\rangle+\ldots+\langle a_n\rangle$. 

If $\KM_*(F)$ denotes the Milnor $\K$-theory ring defined in \cite[\S 1]{Milnor69}, we have a graded surjective ring homomorphism 
\[
f:\KMW_*(F)\to \KM_*(F)
\]
defined by $f([a])=\{a\}$ and $f(\eta)=0$. In fact, $\ker f$ is the principal (two-sided) ideal generated by $\eta$ \cite[Remarque 5.2]{Morel04}. We sometimes refer to $f$ as the \emph{forgetful} homomorphism. On the other hand, let 
\[
H:\KM_*(F)\to \KMW_*(F)
\]
be defined by $H(\{a_1,\ldots,a_n\})=\langle 1,-1\rangle[a_1,\ldots,a_n]=(2+\eta[-1])[a_1,\ldots,a_n]$. Using relation \eqref{relation:steinberg} above, it is easy to check that $H$ is a well-defined graded homomorphism of $\KMW_*(F)$-modules (where $\KM_*(F)$ has the module structure induced by $f$), that we call the \emph{hyperbolic} homomorphism. As $f(\eta)=0$, we see that $fH:\KM_n(F)\to \KM_n(F)$ is the multiplication by $2$ homomorphism.

For any $a\in F^\times$, let $\langle\langle a\rangle\rangle:=\langle a\rangle -1\in \I(F)\subset \GW(F)$ and for any $a_1,\ldots,a_n\in F^\times$ let $\langle\langle a_1,\ldots,a_n\rangle\rangle$ denote the product $\langle\langle a_1\rangle\rangle\cdots\langle\langle a_n\rangle\rangle$ (our notation differs from \cite[\S 2]{Morel04} by a sign). By definition, we have $\langle\langle a_1,\ldots a_n\rangle\rangle\in \I^m(F)$ for any $m\leq n$. In particular, we can define a map
\[
\KMW_n(F)\to \I^n(F)
\]
by $\eta^s[a_1,\ldots,a_{n+s}]\mapsto \langle\langle a_1,\ldots,a_{n+s}\rangle\rangle$ for any $s\in\N$ and any $a_1,\ldots,a_{n+s}\in F^\times$. It follows from \cite[Definition 3.3]{Morel08} and \cite[Lemme 2.3]{Morel04} that this map is a well-defined homomorphism. Moreover, the diagram
\begin{equation}\label{diag:jn}
\begin{gathered}
\xymatrix{\KMW_n(F)\ar[r]\ar[d]_-f & \I^n(F)\ar[d] \\
\KM_n(F)\ar[r]_-{s_n} & \I^n(F)/\I^{n+1}(F)}
\end{gathered}
\end{equation}
where $s_n$ is the map defined in \cite[Theorem 4.1]{Milnor69} is Cartesian by \cite[Th\'eor\`eme 5.3]{Morel04}. 

\subsection{Residues}\label{sec:residues}

Suppose now that $F$ is endowed with a discrete valuation $v:F^\times \to \Z$ with valuation ring $\O_v$, uniformizing parameter $\pi$ and residue field $k(v)$. The following theorem is due to Morel \cite[Theorem 3.15]{Morel08}.
\begin{thm} \label{thm:res}
There exists a unique homomorphism of graded groups
\[
\partial_v^{\pi}:\KMW_*(F)\to \KMW_{*-1}(k(v))
\]
commuting with the product by $\eta$ and such that $\partial_v^{\pi}([\pi,u_2,\ldots,u_n])=[\overline u_2,\ldots,\overline u_n]$ and $\partial_v^{\pi}([u_1,\ldots,u_n])=0$ for any units $u_1,\ldots,u_n\in \O_v^\times$.
\end{thm}
As for Milnor K-theory, $v: F^\times \to \Z$, there also exists a specialization map
\[
s_v^\pi:\KMW_*(F)\to \KMW_*(k(v)),
\]
which is a ring map, and that can be deduced from $\partial_v^\pi$ by the formula 
\begin{equation} \label{eq:sandres}
s_v^\pi(\alpha)=\partial_v^\pi([\pi]\alpha)-[-1]\partial_v^\pi(\alpha)
\end{equation}
(actually, one usually constructs $\partial_v^\pi$ and $s_v^\pi$ together by a trick of Serre, see \cite[Lemma 3.16]{Morel08}). 
\begin{lem}
Both the kernel of $\partial_v^\pi$ and the restriction of $s_v^\pi$ to this kernel are independent of the choice of the uniformizer $\pi$.
\end{lem}
\begin{proof}
If $\pi$ and $u\pi$ are uniformizers, for some $u \in \O_v^\times$, then $\partial_v^{u\pi}=\langle \bar{u}\rangle\partial_v^\pi$. Indeed, by uniqueness in Theorem \ref{thm:res}, it suffices to check this equality on elements of the form $[u_1,\ldots,u_n]$ or $[\pi,u_2,\ldots,u_n]$ with the $u_i$'s units, and on these it is straightforward. Formula \eqref{eq:sandres} then shows that if $\alpha\in \ker(\partial_v)$, then $s_v^\pi(\alpha)=\partial_v^\pi([\pi]\alpha)$, and we compute
\begin{align*}
s_v^{u\pi}(\alpha) &= \partial_v^{u\pi}([u\pi]\alpha) = \langle\bar{u}\rangle \partial_v^\pi([u\pi]\alpha) \\
 &= \langle\bar{u}\rangle \partial_v^\pi([u]\alpha + \langle u\rangle[\pi]\alpha) \\
 &= \langle\bar{u}\rangle \epsilon [\bar{u}] \partial_v^\pi(\alpha) + \partial_v^\pi([\pi]\alpha) && \text{by \cite[Prop.~3.17]{Morel08}} \\
 &= s_v^\pi(\alpha) && \text{since $\alpha \in \ker(\partial_v^\pi)$} \qedhere
\end{align*}
\end{proof}

This lemma allows one to define unramified Milnor-Witt $\K$-theory sheaves as follows. If $X$ is a smooth integral $k$-scheme, any point $x\in X^{(1)}$ defines a discrete valuation $v_x$ for which we can choose a uniformizing parameter $\pi_x$. We then set for any $n\in \Z$
\[
\KMW_n(X):=\ker \Big(\KMW_n\big(k(X)\big)\to \bigoplus_{x\in X^{(1)}}\KMW_{n-1}\big(k(x)\big)\Big)
\]
where the map is induced by the residue homomorphisms $\partial_{v_x}^{\pi_x}$. This kernel is independent of the choices of uniformizers $\pi_x$ by the lemma.
 
Let $i:V\subset X$ be a codimension $1$ closed smooth subvariety defining a valuation $v$ on $k(X)$ with uniformizing parameter $\pi$. We then have $k(V)=k(v)$ and the graded ring map 
\[
s_v^{\pi}:\KMW_*\big(k(X)\big)\to \KMW_*\big(k(V)\big)
\]
restricts to a map $\KMW_*\big(X\big)\to \KMW_*\big(k(V)\big)$ independent of the choice of the uniformizer because $\KMW_n(X) \subseteq \ker(\partial_v)$ so the lemma applies again. Finally, it actually lands in $\KMW_*\big(V\big)$ by \cite[proof of Lemma 2.12]{Morel08} and thus defines a morphism 
\[
i^*:\KMW_n(X)\to \KMW_n(V)
\]
satisfying $i^*(\alpha)=\partial_v^{\pi}([\pi]\alpha)$. Working inductively and locally, the same method shows that we can define pull-back maps $j^*$ for any smooth closed immersion $j:Z\to X$ \cite[p. 21]{Morel08}. On the other hand, it is easy to see that a smooth morphism $h:Y\to X$ induces a homomorphism $h^*:\KMW_n(X)\to \KMW_n(Y)$ and it follows from the standard graph factorization $X\to X\times_k Y\to Y$ that any morphism $f:X\to Y$ gives rise to a pull-back map $f^*$. Thus $X\mapsto \KMW_n(X)$ defines a presheaf $\sKMW_n$ on $\sm{k}$ which turns out to be a Nisnevich sheaf \cite[Lemma 2.12]{Morel08}. We call it the ($n$-th)\emph{Milnor-Witt sheaf}. 

Recall that one can also define residues for Milnor $\K$-theory \cite[Lemma 2.1]{Milnor69} and therefore an unramified Nisnevich sheaf $\sKM_n$ on $\sm{k}$. It is easy to check that both the forgetful and the hyperbolic homomorphisms commute with residue maps. As a consequence, we get morphisms of sheaves $f:\sKMW_n\to \sKM_n$ and $H:\sKM_n\to \sKMW_n$ for any $n\in\Z$ and the composite $fH$ is the multiplication by $2$ map. 

The multiplication map $\KMW_n(F)\times \KMW_m(F)\to \KMW_{n+m}(F)$ induces for any $m,n\in\Z$ a morphism of sheaves
\[
\sKMW_n\times \sKMW_m\to \sKMW_{n+m}
\]
that is compatible with the corresponding product on Milnor K-theory sheaves via the forgetful map.
 
\subsection{Twisting by line bundles} \label{sec:twisting}

We will also need a version of Milnor-Witt K-theory twisted by line bundles, which we now recall following \cite[\S 1.2]{Schmid98} and \cite[\S 5]{Morel08}. 

Let $V$ be a one dimensional vector space over the field $F$. One can consider the group ring $\Z[F^\times]$ and the $\Z[F^\times]$-module $\Z[V^\times]$ where $V^\times=V\setminus 0$. Letting $a \in F^\times$ act by multiplication by $\langle a\rangle$ defines a $\Z$-linear action of the group $F^\times$ on $\GW(F)=\KMW_0$, which therefore extends to a ring morphism $\Z[F^\times] \to \KMW_0(F)$. Thus, we get a $\Z[F^\times]$-module structure on $\KMW_n(F)$ for any $n$, and the action is central (since $\KMW_0(F)$ is central in $\KMW_n(F)$). We then define the $n$-th Milnor-Witt group of $F$ twisted by $V$ as 
\[
\KMW_n(F,V)=\KMW_n(F) \otimes_{\Z[F^\times]} \Z[V^\times]. 
\]
On Nisnevich sheaves, we perform a similar construction. Let $\Z[\gm]$ be the Nisnevich sheaf on $\sm k$ associated to the presheaf $U \mapsto \Z[\O(U)^\times]$. The morphism of sheaves of groups $\gm\to (\sKMW_0)^\times$ defined by $u\mapsto \langle u\rangle$ for any $u\in \O(U)^\times$ extends to a morphism of sheaves of rings $\Z[\gm] \to \sKMW_0$, turning $\sKMW_n$ into a $\Z[\gm]$-module, the action being central. 

Let now $\L$ be a line bundle over a smooth scheme $X$, and let $\Z[\L^\times]$ be the Nisnevich sheafification of $U \mapsto \Z[\L(U)\setminus 0]$. Following \cite[Chapter 5]{Morel08}, we define the Nisnevich sheaf on $\sm{X}$, the category of smooth schemes over $X$, by 
\[
\sKMW_n(\L)= \sKMW_n\otimes_{\Z[\gm]} \Z[\L^\times].
\]
(again, this is the sheaf tensor product).


\section{Transfers in Milnor-Witt $\K$-theory}\label{sec:cohomological}

A very important feature of Milnor-Witt $\K$-theory is the existence of transfers for finite field extensions. They are more subtle than the transfers for Milnor $\K$-theory, and we thus explain them in some details in this section. To avoid technicalities, we suppose that the fields are of characteristic different from $2$.

Recall first that the \emph{geometric} transfers in Milnor-Witt $\K$-theory are defined, for a monogeneous finite field extension $F=L(x)$, using the split exact sequence \cite[Theorem 3.24]{Morel08}
\[
\xymatrix{0\ar[r] & \KMW_n(L)\ar[r] & \KMW_n(L(t))\ar[r]^-\partial & \displaystyle{\bigoplus_{x\in (\A^1_L)^{(1)}} \KMW_{n-1}(L(x))} \ar[r] & 0        }
\] 
where $\partial$ is defined using the residue homomorphisms associated to the valuations corresponding to $x$ and uniformizing parameters the minimal polynomial of $x$ over $L$. If $\alpha\in \KMW_{n-1}(L(x))$, its transfer is defined by choosing a preimage in $\KMW_n(L(t))$ and then applying the residue homomorphism $-\partial_{\infty}$ corresponding to the valuation at infinity (with uniformizing parameter $-\frac 1t$). The corresponding homomorphism $\KMW_{n-1}(L(x))\to \KMW_{n-1}(L)$ is denoted by $\tau_L^F(x)$. It turns out that the geometric transfers do not generalize well to arbitrary finite field extensions $F/L$, and one has to modify them in a suitable way as follows.  

Let again $L\subset F$ be a field extension of degree $n$ generated by $x\in F$. Let $p$ be the minimal polynomial of $x$ over $L$. We can decompose the field extension $L\subset F$ as $L\subset F^{sep}\subset F$, where $F^{sep}$ is the separable closure of $L$ in $F$. If $\mathrm{char}(L)=l\neq 0$, then the minimal polynomial $p$ can be written as $p(t)=p_0(t^{l^m})$ for some $m\in \N$ and $p_0$ separable. Then $F^{sep}=L(x^{l^m})$ and $p_0$ is the minimal polynomial of $x^{l^m}$ over $L$. Following \cite[Definition 4.26]{Morel08}, we set $\omega_0(x):= p_0^\prime(x^{l^m})\in F^\times$ if $l=\mathrm{char}(L)\neq 0$ and $\omega_0(x)=p^\prime(x)\in L^\times$ if $\mathrm{char}(L)=0$. Morel then defines \emph{cohomological} transfers as the composites
\[
\KMW_n(F)\stackrel{\langle \omega_0(x)\rangle }{\longrightarrow} \KMW_n(F)\stackrel{\tau_K^L(x)}{\longrightarrow} \KMW_n(L)
\]
and denotes them by $\Tr_L^F(x)$. If now $F/L$ is an arbitrary finite field extension, we can write
\[
L=F_0\subset F_1\subset\ldots\subset F_m=F
\]
where $F_{i}/F_{i-1}$ is finite and generated by some $x_i\in F_i$ for any $i=1,\ldots,m$. We then set $\Tr^F_L:=\Tr_L^{F_1}(x_1)\circ\ldots \circ \Tr_{F_{m-1}}^{F_m}(x_m)$. It turns out that this definition is independent of the choice of the subfields $F_i$ and of the generators $x_i$ \cite[Theorem 4.27]{Morel08}.

\subsection{Transfers of bilinear forms}

The definition of geometric and cohomological transfers can be recovered from the transfers in Milnor $\K$-theory as well as the Scharlau transfers on bilinear forms as we now explain.

Recall that the Milnor-Witt $\K$-theory group $\KMW_n(L)$ fits into a Cartesian square
\[
\xymatrix{\KMW_n(L)\ar[r]\ar[d] & \I^n(L)\ar[d] \\ 
\KM_n(L)\ar[r]_-{s_n} & \overline {\I^n}(L)}
\]
for any $n\in\Z$, where $\I^n(L)=\W(L)$ for $n<0$, $\KM_n(L)=0$ for $n<0$ and $\overline \I^n(L):=\I^n(L)/\I^{n+1}(L)$ for any $n\in\N$. 

If $L\subset F$ is a finite field extension, then any non-zero $L$-linear homomorphism $f:F\to L$ induces a transfer morphism $f_*:\GW(F)\to \GW(L)$. It follows from \cite[Lemma 1.4]{Milnor73} that this homomorphism induces transfer homomorphisms $f_*:\I^n(F)\to \I^n(L)$ for any $n\in\Z$ and therefore transfer homomorphisms $\overline {f_*}:\overline {\I^n}(F)\to \overline {\I^n}(L)$ for any $n\in\N$. Recall moreover that if $g:F\to L$ is another non zero $K$-linear map, there exists a unit $b\in F^\times$ such that the following diagram commutes 
\begin{equation}\label{equ:trace}
\begin{gathered}
\xymatrix{
\GW(F)\ar[r]^-{\langle b\rangle}\ar[rd]_-{g_*} & \GW(F)\ar[d]^-{f_*} \\
 & \GW(L).
}
\end{gathered}
\end{equation}

Using the split exact sequence of \cite[Theorem 2.3]{Milnor69} and the procedure described above, one can also define transfer morphisms $N_{F/L}:\KM_n(F)\to \KM_n(L)$ (the notation reflects the fact that $N_{F/L}$ coincides in degree $1$ with the usual norm homomorphism).

\begin{lem}
For any non-zero linear homomorphism $f:F\to L$ and any $n\in\N$, the following diagram commutes
\[
\xymatrix{\KM_n(F)\ar[r]^-{N_{F/L}}\ar[d]_-{s_n} & \KM_n(L)\ar[d]^-{s_n} \\
\overline {\I^n}(F)\ar[r]_-{\overline f_*} & \overline {\I^n}(L)}
\]
\end{lem}

\begin{proof}
Observe first that for any $b\in F^\times$, we have $\langle -1,b\rangle\cdot \overline {\I^n}(F)=0$. It follows thus from Diagram \eqref{equ:trace} that $\overline f_*=\overline g_*$ for any non-zero linear homomorphisms $f,g:L\to K$. Now both the transfers for Milnor $\K$-theory and for $\overline {\I^n}$ are functorial, and it follows that we can suppose that the extension $L\subset F$ is monogeneous, say generated by $x\in F$. If $n:=[F:L]$, then $1,x,\ldots, x^{n-1}$ is a $L$-basis of $F$ and we define the $L$-linear map $f:F\to L$ by $f(x^i)=0$ if $i=0,\ldots,n-2$ and $f(x^{n-1})=1$. The result now follows from \cite[Theorem 4.1]{Scharlau72}.
\end{proof}

As a consequence of the lemma, we see that any non-zero linear map $f:F\to L$ induces a transfer homomorphism $f_*:\KMW_n(F)\to \KMW_n(L)$ for any $n\in\Z$. 

\begin{lem}\label{lem:geometrictransfer}
Let $L\subset F$ be a monogeneous field extension of degree $n$ generated by $x\in F$. Then the geometric transfer is equal to the transfer $f_*$ where $f$ is the $L$-linear map defined by $f(x^i)=0$ if $i=0,\ldots,n-2$ and $f(x^{n-1})=1$.
\end{lem}

\begin{proof}
Once again, this follows immediately from Scharlau's reciprocity theorem \cite[Theorem 4.1]{Scharlau72}.
\end{proof}

\begin{lem}
Suppose that $L\subset F$ is a separable field extension, generated by $x\in F$. Then the cohomological transfer coincides with the transfer obtained via the trace map $F\to L$.
\end{lem}

\begin{proof}
This is an immediate consequence of \cite[III, \S 6, Lemme 2]{Serre68}. 
\end{proof}

If the extension $L\subset F$ is purely inseparable, then the cohomological transfer coincides by definition with the geometric transfer. It follows that the cohomological transfer can be computed using trace maps (when the extension is separable), the other homomorphisms described in Lemma \ref{lem:geometrictransfer} (when the extension is inseparable) or a combination of both via the factorization $L\subset F^{sep}\subset F$.

\subsection{Canonical orientations}

In order to properly define the category of finite MW-correspondences, we will have to use differential forms to twist Milnor-Witt $\K$-theory groups. In this section, we collect a few useful facts about orientations of relative sheaves, starting with the general notion of an orientation of a line bundle.

Let $X$ be a scheme, and let $\mathcal N$ be a line bundle over $X$. Recall from \cite[Definition 4.3]{Morel08} that an orientation of $\mathcal N$ is a pair $(\L,\psi)$, where $\L$ is a line bundle over $X$ and $\psi:\L\otimes \L\simeq \mathcal N$ is an isomorphism. Two orientations $(\L,\psi)$ and $(\L^\prime,\psi^\prime)$ are said to be equivalent if there exists an isomorphism $\alpha:\L\to \L^\prime$ such that the diagram
\[
\xymatrix@C=2ex{
\L\otimes \L\ar[rr]^-{\alpha\otimes\alpha}\ar[rd]_-{\psi^\prime} &  & \L^\prime\otimes\L^\prime\ar[ld]^-\psi \\
 & \mathcal N & 
}
\]
commutes. The set of equivalence classes of orientations of $\mathcal N$ is denoted by $\Or(\mathcal N)$. Any invertible element $x$ in the global sections of $\L$ gives a trivialization $\O_X \simeq \L$ sending $1$ to $x$. This trivialization can be considered as an orientation $(\O_X,q_x)$ of $\L$ via the canonical identification $\O_X \otimes \O_X \simeq \O_X$ given by the multiplication. In other words, on sections, $q_x(a \otimes b)=abx$. Clearly, $q_x=q_{x'}$ if and only if $x=u^2 x'$ for some invertible global section $u$.

Let $k\subset L\subset F$ be field extensions such that $F/L$ is finite and $F/k$ and $L/k$ are finitely generated and separable (possibly transcendental). Let $\omega_{F/L}:=\omega_{F/k}\otimes_L\omega_{L/k}^\vee$ be its relative $F$-vector space (according to our conventions, this vector space is the same as $\omega_f$, where $f:\spec F\to \spec L$ is the morphism induced by $L\subset F$). Our goal is to choose for such an extension a canonical orientation of $\omega_{F/L}$.

Suppose first that the extension $L\subset F$ is purely inseparable. In that case, we have a canonical bijection of $\Or(F)$-equivariant sets $\Fr:\Or(\omega_{L/k}\otimes_L F)\to \Or(\omega_{F/k})$ induced by the Frobenius \cite[\S 2.2.4]{Schmid98}. Any choice of a non-zero element $x$ of $\omega_{L/k}$ yields an $L$-linear homomorphism $x^\vee:\omega_{L/k}\to L$ defined by $x^\vee(x)=1$ and an orientation $\Fr(q_x)\in \Or(\omega_{F/k})$. We thus obtain a class in $\Or(\omega_{F/k}\otimes_L\omega_{L/k}^\vee)$ represented by $\Fr(x)\otimes q_{x^\vee}$.

If $x'=ux$ for some $u\in L^\times$, then $(x^\prime)^\vee=u^{-1}x^\vee$ and $\Fr(q_{x'})$ satisfies $\Fr(q_{x'})=q_u \Fr(q_x)\in \Or(\omega_{F/k})$. It follows that $\Fr(q_x)\otimes q_{x^\vee}=\Fr(q_{x'})\otimes q_{(x')^\vee}\in \Or(\omega_{F/k}\otimes_L\omega_{L/k}^\vee)$ and this class is thus independent of the choice of $x\in \omega_{L/k}$. By definition, we have $\omega_{F/L}:=\omega_{F/k}\otimes_L\omega_{L/k}^\vee$ and we therefore get a canonical orientation in $\Or(\omega_{F/L})$.

Suppose next that the extension $L\subset F$ is separable. In that case, the module of differentials $\Omega_{F/L}=0$ and we have a canonical isomorphism $\omega_{F/k}\simeq \omega_{L/k}\otimes_L F$. It follows that $\omega_{F/L}\simeq F$ canonically and we choose the orientation $q_1$ given by $1\in F$ under this isomorphism.
 
We have thus proved the following result.

\begin{lem}\label{lem:orientation}
Let $k\subset L\subset F$ be field extensions such that $F/L$ is finite, $L/k$ and $F/k$ are finitely generated and separable (possibly transcendental). Then there is a canonical orientation of $\omega_{F/L}$.
\end{lem}

\begin{proof}
It suffices to consider $L\subset F^{sep}\subset F$ and the two cases described above. 
\end{proof}

\section{Chow-Witt groups}\label{sec:chowwitt}

Recall the construction of Milnor-Witt K-theory sheaves twisted by line bundles from section \ref{sec:twisting}.

\begin{defin}
For any smooth scheme $X$, any line bundle $\L$ over $X$, any closed subset $Z\subset X$ and any $n\in\N$, we define the \emph{$n$-th Chow-Witt group} (twisted by $\L$, supported on $Z$) by $\chst nZX{\L}:=\H^n_Z(X,\sKMW_n(\L))$. 
\end{defin}
If $\L=\O_X$ (resp. $Z=X$), we omit $\L$ (resp. $Z$) from the notation. Provided the base field $k$ is perfect, the Rost-Schmid complex defined by Morel in \cite[Chapter 5]{Morel08} provides a flabby resolution of $\sKMW_n(\L)$ and we can use it to compute the cohomology of this sheaf. It follows from \eqref{diag:jn} above and \cite[Remark 7.31]{Fasel08a} that this definition coincides with the one given in \cite[D\'efinition 10.2.14]{Fasel08a}. 

The groups $\chst nZX{\L}$ are contravariant in $X$ (and $\L$). If $f:X\to Y$ is a finite morphism between smooth schemes of respective (constant) dimension $d_X$ and $d_Y$, then there is a push-forward map 
\[
f_*:\chst nZX{\omega_{f}\otimes f^*\L}\to \chst {n+d_Y-d_X}{f(Z)}Y{\L}
\]
for any line bundle $\L$ over $Y$ \cite[Corollary 5.30]{Morel08}. More generally, one can define a push-forward map as above for any proper morphism $f:X\to Y$ \cite[Corollaire 10.4.5]{Fasel08a}. Actually, the push-forward map can be slightly generalized if one considers supports. If $f:X\to Y$ is a morphism of smooth schemes and $Z\subset X$ is a closed subscheme which is finite over $W\subset Y$, then we can define a push-forward map
\[
f_*:\chst nZX{\omega_{f}\otimes f^*\L}\to \chst {n+d_Y-d_X}{W}Y{\L}
\]
along the formula given in \cite[p. 125]{Morel08}. Indeed, it suffices to check that the proof of \cite[Corollary 5.30]{Morel08} holds in that case, which is easy.

Observe now that the forgetful morphism of sheaves $\sKMW_n(\L)\to \sKM_n$ yields homomorphisms $\chst nZX{\L}\to \CH^n_Z(X)$ for any $n\in\N$, while the hyperbolic morphism $\sKM_n\to \sKMW_n(\L)$ yields homomorphisms $\CH^n_Z(X)\to \chst nZX{\L}$ for any $n\in\N$. The composite $\sKM_n\to \sKMW_n(\L)\to \sKM_n$ being multiplication by $2$, the composite homomorphism
\[
\CH^n_Z(X)\to \chst nZX{\L}\to \CH^n_Z(X)
\]
is also the multiplication by $2$. Both the hyperbolic and forgetful homomorphisms are compatible with the pull-back and the push-forward maps. Moreover, the total Chow-Witt group of a smooth scheme $X$ is endowed with a ring structure refining the intersection product on Chow groups (i.e. the forgetful homomorphism is a ring homomorphism). More precisely, as for usual Chow groups, there is an external product 
\[
\chst nZX{\L}\times \chst mWX{\mathcal N}\to \chst {m+n}{Z\times W}{X\times X}{p_1^*\L\otimes p_2^*\mathcal N}
\]
commuting to pull-backs and push-forwards. By pulling back along the diagonal, it yields a cup-product
\[
\chst nZX{\L}\times \chst mWX{\mathcal N}\to \chst {m+n}{Z\cap W}X{\L\otimes\mathcal N}
\]
for any $m,n\in\Z$, any line bundles $\L$ and $\mathcal N$ over $X$ and any closed subsets $Z,W\subset X$. It is associative, and its unit is given by the pull-back to $X$ of $\langle 1\rangle\in \sKMW_0(k)=\GW(k)$ \cite[\S 6]{Fasel07}. In general, the product is not commutative. If $\alpha\in \chst nZX{\L}$ and $\beta \in \chst mWX{\mathcal N}$, we have $\alpha\cdot \beta=\langle -1\rangle^{mn}\beta\cdot \alpha$ (under the canonical identification $\mathcal N\otimes \L\simeq \L\otimes \mathcal N$) by \cite[Remark 6.7]{Fasel07}.

For the sake of completeness, recall that Chow-Witt groups satisfy homotopy invariance by \cite[Corollaire 11.3.2]{Fasel08a}.

\subsection{Some useful results}

The goal of this section is to state the analogues of some classical formulas for Chow-Witt groups. Most of them are ``obvious'' in the sense that their proofs are basically the same as for Chow groups. Before stating our first result, let us recall that two morphisms $f:X\to Y$ and $g:U\to Y$ are \emph{Tor-independent} if for every $x\in X$, $y\in Y$ and $u\in U$ such that $f(x)=g(u)=y$ we have $Tor_n^{\O_{Y,y}}(\O_{X,x},\O_{U,u})=0$ for $n\geq 1$.

\begin{prop}[Base change formula]\label{prop:basechange}
Let
\[
\xymatrix{X^\prime\ar[r]^-v\ar[d]_-g & X\ar[d]^-f \\
Y^\prime\ar[r]_-u & Y}
\]
be a Cartesian square of smooth schemes with $f$ proper. Suppose that $f$ and $u$ are Tor-independent. Then $u^*f_*=g_*v^*$.
\end{prop}

\begin{proof}
We first break the square into two Cartesian squares
\[
\xymatrix{X^\prime\ar[r]^-{(g,v)}\ar[d]_-g & Y^\prime\times X\ar[r]^-{p_X}\ar[d]_-{1\times f} & X\ar[d]^-f \\
Y^\prime\ar[r]_-{\Gamma_u} & Y^\prime\times Y\ar[r]_-{p_Y} & Y    }
\] 
where $\Gamma_u$ is the graph of $u$ and the right-hand horizontal morphisms are the respective projections. By \cite[Th\'eor\`eme 12.3.6]{Fasel08a}, we know that the generalized base change formula holds for the right-hand square. Moreover, the morphisms $(1\times f)$ and $\Gamma_u$ are Tor-independent since $f$ and $u$ are. We are thus reduced to show that the formula holds if $u:Y^\prime\to Y$ (and therefore $v:X^\prime\to X$) is a regular embedding of smooth schemes. This follows from \cite[Theorem 2.12]{Asok13}.
\end{proof}

\begin{rem}\label{rem:choice}
In all the formulas involving line bundles as "orientations", we have to specify the isomorphisms we use to identify them. In this paper, we will need the Base change formula in case $u$ is smooth and we then use the following identifications. First, we have a canonical isomorphism $g^*(\omega_{Y^\prime/Y})\simeq \omega_{X^\prime/X}$ given by \cite[Corollary 4.3]{Kunz86}. Since $u$ is smooth, the first fundamental exact sequence
\[
u^*\Omega_{Y/k}\to \Omega_{Y^\prime/k}\to \Omega_{Y^\prime/Y}\to 0
\]
is also exact on the left, and thus yields an isomorphism $u^*\omega_{Y/k}\simeq \omega_{Y^\prime/k}\otimes \omega_{Y^\prime/Y}^\vee$. On the other hand, the smoothness of $v$ and the same argument give an isomorphism $v^*\omega_{X/k}\simeq \omega_{X^\prime/k}\otimes \omega_{X^\prime/X}^\vee$. The base change formula is obtained via the isomorphisms of line bundles
\[
v^*\omega_{X/k}\simeq \omega_{X^\prime/k}\otimes \omega_{X^\prime/X}^\vee\simeq \omega_{X^\prime/k}\otimes g^*(\omega_{Y^\prime/Y})^\vee
\]
and
\[
u^*\omega_{Y/k}\simeq \omega_{Y^\prime/k}\otimes \omega_{Y^\prime/Y}^\vee.
\]
\end{rem}

\begin{rem}\label{rem:noncartesian}
There is no need for $f$ to be proper in the above proposition, as long as we consider supports which are proper over the base. More precisely, suppose that we have a Cartesian square
\[
\xymatrix{X^\prime\ar[r]^-v\ar[d]_-g & X\ar[d]^-f \\
Y^\prime\ar[r]_-u & Y}
\]
of smooth schemes with $f$ and $u$ Tor-independent. Let $M\subset X$ be a closed subset such that the composite morphism $M\subset X\stackrel f\to Y$ is proper (here $M$ is endowed with its reduced scheme structure). Then the formula $u^*f_*=g_*v^*$ holds for any $\alpha\in \chst nMX{\omega_{X/Y}\otimes f^*\L}$. The proof is the same as the proof of the proposition, taking supports into account.
\end{rem}

\begin{coro}[Projection formula]\label{cor:pformula}
Let $m,n\in\N$ and let $\L$, $\mathcal N$ be line bundles over $Y$. Let $Z\subset X$ and $W\subset Y$ be closed subsets. If $f:X\to Y$ is a proper morphism of (constant) relative dimension $c\in \Z$, we have 
\[
f_*(\alpha)\cdot \beta=f_*(\alpha\cdot f^*(\beta))
\]
for any $\alpha\in \chst mZX{\omega_{f}\otimes f^*\L}$ and $\beta\in \chst nWY{\mathcal N}$.
\end{coro}

\begin{proof}
It suffices to use the base change formula on the following Cartesian square:
\[
\xymatrix@C=4em{
X\ar[r]^-{(1\times f)\Delta_X}\ar[d]_-f & X\times Y\ar[d]^-{(f\times 1)} \\
Y\ar[r]_-{\Delta_Y} & Y\times Y
}
\]
to get $\Delta_Y^*(f\times 1)_*(\alpha\times\beta)=f_*(\alpha\cdot f^*\beta)$ and $\Delta_Y^*(1\times f)_*(\beta\times \alpha)=f_*(f^*\beta\cdot \alpha)$. The result now follows from the equality $(f\times 1)_*(\alpha\times\beta)=f_*\alpha\times \beta$ (\cite{Calmes17}).
\end{proof}

\begin{rem}
To obtain a formula for the left-module structure, we first observe that the push-forward makes sense for line bundles of the form $\omega_f\otimes f^*\mathcal L$. Now, $f^*(\beta)\cdot \alpha$ is a cycle in $\chst {m+n}{Z\cap f^{-1}(W)}X{f^*\mathcal N\otimes \omega_{f}\otimes f^*\L}$, which we have to transform in a cycle in $\chst {m+n}{Z\cap f^{-1}(W)}X{ \omega_{f}\otimes f^*(\L\otimes \mathcal N)}$. For this, we can use the isomorphism
\[
f^*\mathcal N\otimes \omega_{f}\otimes f^*\L\simeq \omega_{f}\otimes f^*\L\otimes f^*\mathcal N\simeq \omega_{f}\otimes f^*(\L\otimes \mathcal N)
\]
obtained from the switch isomorphism and the canonical isomorphism $f^*\L\otimes f^*\mathcal N\simeq f^*(\L\otimes \mathcal N)$. Under this identification, we have 
\[
f^*(\beta)\cdot \alpha=\langle (-1)^{mn}\rangle (\alpha\cdot f^*(\beta))=(\alpha\cdot f^*(\beta))\langle (-1)^{mn}\rangle.
\]
Now, 
\[
f_*(f^*(\beta)\cdot \alpha)=f_*((\alpha\cdot f^*(\beta))\langle (-1)^{mn}\rangle)=f_*(\alpha\cdot f^*(\beta\cdot \langle (-1)^{mn}\rangle))
\]
and the projection formula yields 
\[
f_*(\alpha\cdot f^*(\beta\cdot \langle (-1)^{mn}\rangle))=f_*(\alpha)\cdot \beta\cdot \langle (-1)^{mn}\rangle.
\]
The latter is seen as a cycle with orientation in the line bundle $\L\otimes \mathcal N$. Using the switch isomorphism $\L\otimes \mathcal N\simeq \mathcal N\otimes \L$ and the commutation formula once again, we get
\[
f_*(\alpha)\cdot \beta\cdot \langle (-1)^{mn}\rangle=\langle (-1)^{(m-c)n}\rangle  \beta\cdot \langle (-1)^{mn}\rangle\cdot f_*(\alpha)=\langle (-1)^{cn}\rangle \beta\cdot f_*(\alpha).
\]
Thus, we see that the projection formula reads as $f_*(f^*(\beta)\cdot \alpha)=\langle (-1)^{cn}\rangle \beta\cdot f_*(\alpha)$. At the risk of being annoying, let us stress once again that this formula depends on the choice of the two above isomorphisms of line bundles. 
\end{rem}

\begin{lem}[Flat excision]
Let $f:X\to Y$ be a flat morphism of smooth schemes. Let $V\subset Y$ be a closed subset such that the morphism $f^{-1}(V)\to V$ induced by $f$ is an isomorphism. Then the pull-back morphism
\[
f^*:\chst iV{Y}{\L}\to \chst i{f^{-1}(V)}X{f^*\L}
\]
is an isomorphism for any $i\in\N$ and any line bundle $\L$ over $Y$.
\end{lem}

\begin{proof}
As said in Section \ref{sec:chowwitt}, Chow-Witt groups can be computed using the flabby resolution provided by the Rost-Schmid complex of \cite[Chapter 5]{Morel08}, which coincide with the complex considered in \cite[D\'efinition 10.2.7]{Fasel08a}. Now Chow-Witt groups with supports are obtained by considering the subcomplex of points supported on a certain closed subset. The lemma follows now from the fact that in our case $f^*$ induces (by definition) an isomorphism of complexes. 
\end{proof}

We now consider the problem of describing the cohomology of $X\times \gm$ with coefficients in $\sKMW_j$ (for $j\in\Z$) in terms of the cohomology of $X$. First observe that the pull-back along the projection $p:X\times \gm\to X$ endows the cohomology of $X\times \gm$ with the structure of a module over the cohomology of $X$. Let $t$ be a parameter of $\gm$. The class $[t]$ in $\KMW_1(k(t))$ actually lives in its subgroup $\sKMW_1(\gm)$ since it clearly has trivial residues at all closed points of $\gm$. 
Pulling back to $X\times \gm$ along the projection to the second factor, we get an element in $\sKMW_1(X\times \gm)$ that we still denote by $[t]$.

\begin{lem}\label{lem:explicitcontraction}
For any $i\in\N$, any $j\in \Z$ and any smooth scheme $X$ over $k$, we have 
\[
\H^{i}(X\times \gm,\sKMW_j)=\H^i(X,\sKMW_j)\oplus \H^{i}(X,\sKMW_{j-1})\cdot [t].
\]
\end{lem}

\begin{proof}
The long exact sequence associated to the open immersion $X\times \gm\subset X\times \A^1_k$ reads as
\[
\cdots\to \H^i(X\times\A^1_k,\sKMW_j)\to \H^i(X\times\gm,\sKMW_j)\stackrel\partial\to \H^{i+1}_{X\times \{0\}}(X\times \A^1,\sKMW_j)\to \cdots
\] 
By homotopy invariance, the pull-back along the projection to the first factor $X\times\A^1_k\to X$ induces an isomorphism $\H^i(X,\sKMW_j)\to \H^i(X\times\A^1_k,\sKMW_j)$. Pulling-back along the morphism $X\to X\times\gm$ defined by $x\mapsto (x,1)$ we get a retraction of the composite homomorphism
\[
\H^i(X,\sKMW_j)\to \H^i(X\times\A^1_k,\sKMW_j)\to \H^i(X\times\gm,\sKMW_j)
\]
and it follows that the long exact sequence splits into short split exact sequences
\[
0\to \H^i(X\times\A^1_k,\sKMW_j)\to \H^i(X\times\gm,\sKMW_j)\stackrel\partial\to \H^{i+1}_{X\times \{0\}}(X\times \A^1,\sKMW_j)\to 0
\] 
Now the push-forward homomorphism (together with the obvious trivialization of the normal bundle to $X\times\{0\}$ in $X\times\A^1_k$) yields an isomorphism \cite[Remarque 10.4.8]{Fasel08a}
\[
\iota:\H^i(X,\sKMW_{j-1})\to \H^{i+1}_{X\times \{0\}}(X\times \A^1,\sKMW_j)
\]
and it suffices then to check that the composite 
\[
\H^i(X,\sKMW_{j-1})\stackrel{[t]}\to \H^i(X,\sKMW_{j})\to \H^i(X\times \gm,\sKMW_{j})\stackrel{\iota^{-1}\partial}\to \H^i(X,\sKMW_{j-1}) 
\]
is an isomorphism to conclude. This follows essentially from \cite[Proposition 3.17. 2)]{Morel08}.
\end{proof}

\begin{rem} \label{rem:ttozero}
By pull-back along the morphism $\spec k \to \gm$ sending the point to $1$, the element $[t] \in \sKMW_1(\gm)$ maps to $[1]=0$ in $\sKMW_1(k)$. Therefore this pull-back gives the splitting of $\H^i(X,\sKMW_j)$ in the decomposition above. 
\end{rem}


\section{Finite MW-correspondences}

\subsection{Admissible subsets}\label{sec:fGW}

Let $X$ and $Y$ be smooth schemes over $\spec k$ and let $T\subset X\times Y$ be a closed subset. Any irreducible component of $T$ maps to an irreducible component of $X$ through the projection $X \times Y \to X$. 
\begin{defin}
If, when $T$ is endowed with its reduced structure, this map is finite and surjective for every irreducible component of $T$, we say that $T$ is an \emph{admissible subset} of $X\times Y$. 
We denote by $\Adm(X,Y)$ the set of admissible subsets of $X\times Y$, partially ordered by inclusions. As usual, we sometimes consider $\Adm(X,Y)$ as a category.
\end{defin}

\begin{rem} \label{rem:admissible}
Since the empty set has no irreducible component, it is admissible.
An irreducible component of an admissible subset is clearly admissible, and the irreducible admissible subsets are minimal (non-trivial) elements in $\Adm (X,Y)$. Furthermore, any finite union of admissible subsets is admissible.
\end{rem}

\begin{lem} \label{lem:admissiblesheaf}
If $f:X' \to X$ is a morphism between smooth schemes, then $T \mapsto (f \times \id_Y)^{-1}(T)$ defines a map $\Adm(X,Y) \to \Adm(X',Y)$. Furthermore, the presheaf $U \mapsto \Adm(U,Y)$ thus defined is a sheaf for the Zariski topology. 
\end{lem}
\begin{proof}
Finiteness and surjectivity are stable by base change by \cite[6.1.5]{EGAI} and \cite[3.5.2]{EGAII}, so the map is well-defined. The injectivity condition in the sheaf sequence is obvious. To prove the exactness in the middle, being closed is obviously a Zariski local property, so the union of the closed subsets in the covering defines a global closed subset. Both finiteness and surjectivity are properties that are Zariski local on the base, so this closed subset is admissible. 
\end{proof}

If $Y$ is equidimensional, $d=\dim Y$ and $p_Y:X\times Y\to Y$ is the projection, we define a covariant functor 
\[
\Adm(X,Y)\to {\mathcal Ab}
\]
by associating to each admissible subset $T\in \Adm(X,Y)$ the group $\chst dT{X\times Y}{p_Y^*\omega_{Y/k}}$ and to each morphism $T^\prime\subset T$ the extension of support homomorphism 
\[
\chst d{T^\prime}{X\times Y}{p_Y^*\omega_{Y/k}}\to \chst dT{X\times Y}{p_Y^*\omega_{Y/k}}
\]
and, using that functor, we set
\[
\cor k(X,Y)=\varinjlim_{T\in \Adm(X,Y)}\chst dT{X\times Y}{p_Y^*\omega_{Y/k}}.
\]
If $Y$ is not equidimensional, then $Y=\coprod_jY_j$ with each $Y_j$ equidimensional and we set
\[
\cor k(X,Y)=\prod_{j} \cor k(X,Y_j).
\]
By additivity of Chow-Witt groups, if $X =\coprod_iX_i$ and $Y=\coprod_j Y_j$ are the respective decompositions of $X$ and $Y$ in irreducible components, we have
\[
\cor k(X,Y)=\prod_{i,j} \cor k(X_i,Y_j).
\]

\begin{nota}\label{nota:omega}
In the sequel, we will simply write $\omega_Y$ in place of $p_Y^*\omega_{Y/k}$. In case of possible confusion, we will switch back to the complete notation.
\end{nota}

\begin{exem}\label{ex:basic}
Let $X$ be a smooth scheme of dimension $d$. Then 
\[
\cor k({\spec k},X)=\bigoplus_{x\in X^{(d)}} \chst d{\{x\}}X{\omega_X}=\bigoplus_{x\in X^{(d)}} \GW(k(x),\omega_{k(x)/k}).
\]
On the other hand, $\cor k(X,\spec k)=\ch 0X=\sKMW_0(X)$ for any smooth scheme $X$.
\end{exem}

The group $\cor k(X,Y)$ admits an alternate description which is often useful. Let $X$ and $Y$ be smooth schemes, with $Y$ equidimensional. For any closed subscheme $T\subset X\times Y$ of codimension $d=\dim Y$, we have an inclusion 
\[
\chst dT{X\times Y}{{\omega_{Y}}}\hspace{1ex}\subset \bigoplus_{x\in (X\times Y)^{(d)}}\hspace{-2ex}\KMW_0(k(x),\omega_x\otimes ({\omega_{Y}})_x).
\]
and it follows that
\[
\cor k(X,Y)=\hspace{-3ex}\bigcup_{T\in \Adm(X,Y)}\hspace{-3ex}\chst dT{X\times Y}{{\omega_{Y}}}\hspace{1ex}\subset \hspace{-2ex}\bigoplus_{x\in (X\times Y)^{(d)}}\hspace{-3ex}\KMW_0(k(x),\omega_x\otimes ({\omega_{Y}})_x).
\]
In general, the inclusion $\cor k(X,Y)\subset \bigoplus_{x\in (X\times Y)^{(d)}}\KMW_0(k(x),\omega_x\otimes ({\omega_{Y}})_x)$ is strict as shown by Example \ref{ex:basic}. As an immediate consequence of this description, we see that the map
\[
\chst dT{X\times Y}{{\omega_{Y}}}\to \cor k(X,Y)
\]
is injective for any $T\in \Adm(X,Y)$. 

If $U$ is an open subset of a smooth scheme $V$, since an admissible subset $T \in \Adm(V,Y)$ intersects with $U \times Y$ as an admissible subset by Lemma \ref{lem:admissiblesheaf}, the pull-backs along $V \times Y \to U \times Y$ on Chow-Witt groups with support induce at the limit a map $\cor k(V,Y) \to \cor k(U,Y)$.

\begin{lem} \label{lem:restrictioninj}
This map is injective.
\end{lem}

\begin{proof}
We can assume $Y$ equidimensional, of dimension $d$. Let $Z=(V\setminus U)\times Y$. Let moreover $T\subset V\times Y$ be an admissible subset. Since $T$ is finite and surjective over $V$, the subset $Z\cap T$ is of codimension at least $d+1$ in $V\times Y$, which implies that $\chst d{Z\cap T}{V\times Y}{\omega_{Y}}=0$. The long exact sequence of localization with support then shows that the homomorphism
\[
\chst d{T}{V\times Y}{\omega_Y}\to \chst d{T\cap (U\times Y)}{U\times Y}{\omega_Y}
\]
is injective. On the other hand, we have a commutative diagram
\[
\xymatrix{
\chst d{T}{V\times Y}{\omega_Y}\ar@{^{(}->}[r]\ar@{^{(}->}[d] & \chst d{T\cap (U\times Y)}{U\times Y}{\omega_Y}\ar@{^{(}->}[d] \\
\cor k(V,Y)\ar[r] & \cor k(U,Y)
}
\]
with injective vertical maps. Since any $\alpha\in \cor k(V,X)$ comes from the group $\chst d{T}{V\times Y}{\omega_Y}$ for some $T\in \Adm(X,Y)$, the homomorphism $\cor k(V,Y)\to \cor k(U,Y)$ is injective. 
\end{proof}

\begin{defin}\label{def:support}
Let $\alpha\in \cor k(X,Y)$, where $X$ and $Y$ are smooth. If $Y$ is equidimensional, let $d=\dim(Y)$. The \emph{support} of $\alpha$ is the closure of the set of points $x\in (X\times Y)^{(d)}$ such that the component of $\alpha$ in $\KMW_0\big(k(x),\omega_x\otimes ({\omega_{Y}})_x\big)$ is nonzero. 
If $Y$ is not equidimensional, then we define the support of $\alpha$ as the union of the supports of the components appearing in the equidimensional decomposition. 
\end{defin}

\begin{lem} \label{lem:supportadmis}
The support of an $\alpha \in \cor k(X,Y)$ is an admissible subset, say $T$, and $\alpha$ is then in the image of the inclusion $\chst dT{X\times Y}{\omega_{Y}} \subset \cor k(X,Y)$.
\end{lem}

\begin{proof}
By definition of $\cor k(X,Y)$ as a direct limit, the support of $\alpha$ is included in some admissible subset $T \in \Adm(X,Y)$. Being finite and surjective over $X$, any irreducible component $T_i$ of $T$ is of codimension $\dim Y$ in $X \times Y$. Therefore the support of $\alpha$ is exactly the union of all $T_i$ such that the component of $\alpha$ on the generic point of $T_i$ is non-zero. This is an admissible subset by Remark \ref{rem:admissible}. 

To obtain the last part of the statement, let $S \subset T$ be the support of $\alpha$ and let $U$ be the open subscheme $X \times Y \setminus S$. Consider the commutative diagram 
\[
\xymatrix@C=3ex{
 \chst dT {X \times Y}{\omega_{Y}} \ar[r] \ar[d] & \chst d{T \setminus S} {U}{(\omega_{Y})_{|U}} \ar[d] \\
 \hspace{-5ex}{\displaystyle\bigoplus_{x\in (X\times Y)^{(d)}\cap T}}\hspace{-4.5ex}\KMW_0(k(x),\omega_x\otimes ({\omega_{Y}})_x) \ar[r] & \hspace{-3ex}{\displaystyle\bigoplus_{x\in U^{(d)}\cap T}}\hspace{-2.5ex}\KMW_0(k(x),\omega_x\otimes ({\omega_{Y}})_x)
}
\]
with injective vertical maps (still for dimensional reasons). By definition of the support, $\alpha$ maps to zero in the lower right group, so it maps to zero in the upper right one. Therefore, it comes from the previous group in the localization exact sequence for Chow groups with support, and this group is $\chst dS{X\times Y}{\omega_{Y}}$. 
\end{proof}

\begin{exem}
In contrast with usual correspondences, an element of $\cor k(X,Y)$ cannot be in general written as the sum of elements with irreducible support. Indeed, let $X=Y=\A^1$. Let moreover $T_1=\{x=y\}\subset \A^1\times \A^1\}$ and $T_2=\{x=-y\}\subset \A^1\times \A^1$. Then $T_1\cap T_2=0\in \A^1\times \A^1$. We can consider $\langle x\rangle\otimes \overline{(x-y)}$ in $\KMW_0(k(T_1),{\mathfrak m}_{T_1}/{\mathfrak m}_{T_1}^2)$ and $\langle x\rangle\otimes \overline{(x+y)}$ in $\KMW_0(k(T_2),{\mathfrak m}_{T_2}/{\mathfrak m}_{T_1}^2)$. The residue of the first one is 
\[
\langle 1\rangle\otimes \overline{(x-y)}\wedge \overline x=\langle 1\rangle\otimes -\overline y\wedge \overline x=\langle -1\rangle \otimes \overline y\wedge \overline x 
\]
in $\KMW_{-1}(k,\wedge^2{\mathfrak m}_0/{\mathfrak m}_0^2)$, while the residue of the second one is
\[
\langle 1\rangle\otimes \overline{(x+y)}\wedge \overline x=\langle 1\rangle \otimes \overline y\wedge \overline x
\]
in the same group. As $\langle -1\rangle+\langle 1\rangle=0\in \KMW_{-1}(k)$, it follows that the sum of the two elements above define an unramified element in $ \chs 1{T_1\cup T_2}{\A^1\times \A^1}$. As the canonical sheaf $\omega_{\A^1}$ is trivial, we obtain an element of $\cor k(\A^1,\A^1)$ which is not the sum of elements with irreducible support (each component is ramified).
\end{exem}

Let $\alpha \in \cor k(X,Y)$ with support $T$ be restricted to an element denoted by $\alpha_{|U} \in \cor k(U,Y)$.  
\begin{lem} \label{lem:supporttoopen}
The support of $\alpha_{|U}$ is $T \cap U$, in other words the image of $T$ by the map $\Adm(X,Y) \to \Adm(U,Y)$. 
\end{lem}
\begin{proof}
It is straightforward from the definition of the support.
\end{proof}

\subsection{Composition of finite MW-correspondences}

Let $X$, $Y$ and $Z$ be smooth schemes of respective dimensions $d_X,d_Y$ and $d_Z$, with $X$ and $Y$ connected. Let $V\in \Adm(X,Y)$ and $T\in \Adm(Y,Z)$ be admissible subsets. Consider the following commutative diagram where all maps are canonical projections:
\begin{equation}\label{eqn:composition}
\begin{gathered}
\xymatrix{
X\times Z \ar@/^2em/[rrrd]^-{r_Z}\ar@/_1.3em/[rddd]_-{p_X} & & & \\
 & X\times Y\times Z\ar[r]^-{q_{YZ}}\ar[d]_-{p_{XY}}\ar[lu]_-{q_{XZ}} & Y\times Z\ar[r]^-{q_Z}\ar[d]^-{p_Y} & Z \\
 & X\times Y\ar[r]_-{q_Y}\ar[d]_-p & Y & \\
 & X & & 
}
\end{gathered}
\end{equation}
We have homomorphisms
\[
(p_{XY})^*:\chst {d_Y}V{X\times Y}{{\omega_{Y}}}\to \chst {d_Y}{(p_{XY})^{-1}V}{X\times Y\times Z}{(p_{XY})^*{\omega_{Y}}}
\]
and
\[
(q_{YZ})^*:\chst {d_Z}T{Y\times Z}{{\omega_{Z}}}\to \chst {d_Z}{(q_{YZ})^{-1}T}{X\times Y\times Z}{(q_{YZ})^*{\omega_{Z}}}.
\]
Let $M=(p_{XY})^{-1}V\cap (q_{YZ})^{-1}T$, endowed with its reduced structure. It follows from \cite[Lemmas 1.4 and 1.6]{Mazza06} that every irreducible component of $M$ is finite and surjective over $X$. As a consequence, the map $M\to q_{XZ}(M)$ is finite and the push-forward
\[
(q_{XZ})_*:\chst {d_Y+d_Z}{M}{X\times Y\times Z}{\omega_{X\times Y\times Z} \otimes q_{XZ}^*\L}\to \chst {d_Z}{q_{XZ}(M)}{X\times Z}{\omega_{X\times Z}\otimes \L}
\]
is well-defined for any line bundle $\L$ over $X\times Z$. In particular for $\L=p_X^*\omega_{X/k}^\vee$, we get a push-forward map
\[
\chst {d_Y+d_Z}{M}{X\times Y\times Z}{\omega_{X\times Y\times Z} \otimes (p_{XY})^*p^*\omega_{X/k}^\vee}\stackrel{(q_{XZ})_*}\to \chst {d_Z}{q_{XZ}(M)}{X\times Z}{\omega_{X\times Z}\otimes p_X^*\omega_{X/k}^\vee}.
\]
\begin{lem}
We have canonical isomorphisms 
\begin{equation}\label{eqn:side1}
\omega_{X\times Y\times Z} \otimes (p_{XY})^*p^*\omega_{X/k}^\vee\simeq (p_{XY})^*{\omega_{Y}}\otimes (q_{YZ})^*{\omega_{Z}}.
\end{equation}
and
\begin{equation}\label{eqn:side2}
\omega_{X\times Z}\otimes p_X^*\omega_{X/k}^\vee\simeq \omega_Z
\end{equation}
\end{lem}

\begin{proof}
We have a canonical isomorphism 
\[
\omega_{X\times Y\times Z}\simeq (p_{XY})^*p^*\omega_{X/k}\otimes (p_{XY})^*q_{Y}^*\omega_{Y/k}\otimes (q_{YZ})^*q_Z^*\omega_{Z/k}. 
\]
Next, we can consider $(-1)^{d_Xd_Y+d_Xd_Z}$ the switch isomorphism
\[
(p_{XY})^*p^*\omega_{X/k}\otimes (p_{XY})^*q_{Y}^*\omega_{Y/k}\otimes (q_{YZ})^*q_Z^*\omega_{Z/k}\simeq (p_{XY})^*q_{Y}^*\omega_{Y/k}\otimes (q_{YZ})^*q_Z^*\omega_{Z/k}\otimes (p_{XY})^*p^*\omega_{X/k} 
\]
Tensoring the two above isomorphisms by $\id_{\omega_{X/k}^\vee}$ and composing with the (pull-back of the) isomorphism $\omega_{X/k}\otimes \omega_{X/k}^\vee\simeq \O_X$, we obtain the isomorphism (\ref{eqn:side1}). For the second isomorphism, we use the canonical isomorphism
\[
\omega_{X\times Z}\simeq p_X^*\omega_{X/k}\otimes r_Z^*\omega_{Z/k}
\]
together with $(-1)^{d_Xd_Z}$ times the switch isomorphism 
\[
p_X^*\omega_{X/k}\otimes r_Z^*\omega_{Z/k}\simeq r_Z^*\omega_{Z/k}\otimes  p_X^*\omega_{X/k}
\]
and the (pull-back of the) canonical isomorphism $\omega_{X/k}\otimes \omega_{X/k}^\vee\simeq \O_X$.
\end{proof}

As a consequence, if we have cycles  $\beta\in  \chst {d_Y}V{X\times Y}{\omega_Y}$ and $\alpha\in \chst {d_Z}T{Y\times Z}{\omega_Z}$ the expression
\[
\alpha\circ \beta:=(q_{XZ})_*[(q_{YZ})^*\beta\cdot (p_{XY})^*\alpha]
\] 
is well-defined. Moreover, it follows from \cite[Lemma 1.7]{Mazza06} that $q_{XZ}(M)$ is an admissible subset of $X\times Z$.
All the above homomorphisms commute with extension of supports, and therefore we get a well-defined composition 
\[
\circ:\cor k(X,Y)\times \cor k(Y,Z)\to \cor k(X,Z).
\]

\begin{rem}\label{rem:everythingcommutes}
Note that in the definition of the composition we could have considered the product $(p_{XY})^*\alpha\cdot (q_{YZ})^*\beta$ in place of the product $(q_{YZ})^*\beta\cdot (p_{XY})^*\alpha$. We claim that this is the same. Indeed, note that 
\[
(p_{XY})^*\alpha\cdot (q_{YZ})^*\beta=\langle (-1)^{d_Yd_Z}\rangle (q_{YZ})^*\beta\cdot (p_{XY})^*\alpha
\]
by \cite[Remark 6.7]{Fasel07}. Here, let us again stress that this comparison is obtained using the switch isomorphism $(p_{XY})^*\omega_Z\otimes (q_{YZ})^*\omega_Y\simeq (q_{YZ})^*\omega_Y\otimes (p_{XY})^*\omega_Z$. However, the canonical isomorphism $(p_{XY})^*\omega_Z\otimes (q_{YZ})^*\omega_Y\simeq (q_{YZ})^*\omega_Y\otimes (p_{XY})^*\omega_Z$ is $(-1)^{d_Yd_Z}$ times the switch isomorphism, showing that 
\[
(p_{XY})^*\alpha\cdot (q_{YZ})^*\beta= (q_{YZ})^*\beta\cdot (p_{XY})^*\alpha
\]
where the comparison is obtained via the canonical isomorphism $(p_{XY})^*\omega_Z\otimes (q_{YZ})^*\omega_Y\simeq (q_{YZ})^*\omega_Y\otimes (p_{XY})^*\omega_Z$. The claim follows.
\end{rem}

The proof that the composition we defined is associative follows essentially from the fact that the intersection product is associative, but there are some subtleties involved when dealing with the required line bundles so we write it for the sake of completeness. 

\begin{lem}\label{lem:associative}
The composition of finite MW-correspondences is associative.
\end{lem}

\begin{proof}
We will use the following Cartesian square where all morphisms are projections onto the respective factors:
\begin{equation}\label{eq:assoc1}
\xymatrix{ & & X\times Y\ar[r]^-{q_Y} & Y & \\
 & X\times T & X\times Y\times T\ar[r]^-{q_{YT}}\ar[l]_-{q_{XT}}\ar[u]^-{s_{XY}} & Y\times T\ar[u]^-{s_{Y}} &  \\
Z\times T\ar[d]_-{p_Z} & X\times Z\times T\ar[d]_-{p_{XZ}}\ar[u]^-{p_{XT}}\ar[l]_-{r_{ZT}} & X\times Y\times Z\times T\ar[r]^-{q_{YZT}}\ar[d]_-{p_{XYZ}}\ar[l]_-{q_{XZT}}\ar[u]^-{p_{XYT}} & Y\times Z\times T\ar[r]^-{q_{ZT}}\ar[d]_-{p_{YZ}}\ar[u]^-{p_{YT}} & Z\times T \ar[d]_-{p_Z}  \\
Z & X\times Z\ar[d]_-{p_X}\ar[l]_-{r_Z} & X\times Y\times Z\ar[r]^-{q_{YZ}}\ar[d]_-{p_{XY}}\ar[l]_-{q_{XZ}} & Y\times Z\ar[r]^-{q_Z}\ar[d]_-{p_Y} & Z  \\
 & X & X\times Y\ar[r]^-{q_Y}\ar[l]_-{q_X} & Y. & }
\end{equation}

Now, let $\alpha\in \cor k(X,Y)$, $\beta\in \cor k(Y,Z)$ and $\gamma\in \cor k(Z,T)$. In our computation, we treat them as elements of some Chow-Witt group, and consider their push-forwards and pull-backs as usual. We omit the relevant line bundles, all our choices being canonical and already mentioned in the previous paragraphs. The composite $\beta\circ\alpha$ is represented by the cycle $(q_{XZ})_*(p_{XY}^*\alpha\cdot q_{YZ}^*\beta)$, while the composite $\gamma\circ (\beta\circ\alpha)$ is given by
\[
(p_{XT})_*\left( p_{XZ}^*(q_{XZ})_*(p_{XY}^*\alpha\cdot q_{YZ}^*\beta)\cdot r_{ZT}^*\gamma\right).
\]
Using the base change formula (Proposition \ref{prop:basechange}, with the canonical isomorphisms of Remark \ref{rem:choice}), we obtain
\[
p_{XZ}^*(q_{XZ})_*(p_{XY}^*\alpha\cdot q_{YZ}^*\beta)=(q_{XZT})_*p_{XYZ}^*(p_{XY}^*\alpha\cdot q_{YZ}^*\beta).
\]
Next, we can use the projection formula (Corollary \ref{cor:pformula}) to get
\[
(q_{XZT})_*p_{XYZ}^*(p_{XY}^*\alpha\cdot q_{YZ}^*\beta)\cdot r_{ZT}^*\gamma=(q_{XZT})_*\left(p_{XYZ}^*(p_{XY}^*\alpha\cdot q_{YZ}^*\beta)\cdot q_{XZT}^*r_{ZT}^*\gamma\right).
\]
Since the product on Chow-Witt groups is associative, it follows that the composite $\gamma\circ (\beta\circ\alpha)$ is the push-forward along the projection $X\times Y\times Z\times T\to X\times T$ of the product of the pull-backs of $\alpha,\beta,\gamma$ along the respective projections. 

We now turn to the computation of $(\gamma\circ \beta)\circ\alpha$. The composite $\gamma\circ\beta$ is given by $(p_{YT})_*(p_{YZ}^*\beta\cdot q_{ZT}^*\gamma)$, while $(\gamma\circ \beta)\circ\alpha$ is of the form
\[
(q_{XT})_*\left(s_{XY}^*\alpha\cdot q_{YT}^*(p_{YT})_*(p_{YZ}^*\beta\cdot q_{ZT}^*\gamma)\right).
\]
Using the base change formula once again, we obtain
\[
q_{YT}^*(p_{YT})_*(p_{YZ}^*\beta\cdot q_{ZT}^*\gamma)=(p_{XYT})_*q_{YZT}^*(p_{YZ}^*\beta\cdot q_{ZT}^*\gamma).
\]
Here, $(p_{XYT})_*q_{YZT}^*(p_{YZ}^*\beta\cdot q_{ZT}^*\gamma)$ is a cycle of codimension $d_T$ (the dimension of $T$) with coefficients in the line bundle $\omega_T$.
Using Remark \ref{rem:everythingcommutes}, we see that 
\[
s_{XY}^*\alpha\cdot (p_{XYT})_*q_{YZT}^*(p_{YZ}^*\beta\cdot q_{ZT}^*\gamma)= (p_{XYT})_*q_{YZT}^*(p_{YZ}^*\beta\cdot q_{ZT}^*\gamma)\cdot s_{XY}^*\alpha
\]
and we can use the projection formula (Corollary \ref{cor:pformula}) to obtain 
\[
(p_{XYT})_*q_{YZT}^*(p_{YZ}^*\beta\cdot q_{ZT}^*\gamma)\cdot s_{XY}^*\alpha=(p_{XYT})_*\left(q_{YZT}^*(p_{YZ}^*\beta\cdot q_{ZT}^*\gamma)\cdot p_{XYT}^*s_{XY}^*\alpha\right).
\]
Now, $q_{YZT}^*(p_{YZ}^*\beta\cdot q_{ZT}^*\gamma)$ is a cycle of codimension $d_Z+d_T$ with coefficients in $\omega_Z\otimes \omega_T$ and $p_{XYT}^*s_{XY}^*\alpha$ is a cycle of codimension $d_Y$ with coefficients in $\omega_Y$. Applying Remark \ref{rem:everythingcommutes} once again, we obtain
\[
q_{YZT}^*(p_{YZ}^*\beta\cdot q_{ZT}^*\gamma)\cdot p_{XYT}^*s_{XY}^*\alpha=p_{XYT}^*s_{XY}^*\alpha\cdot q_{YZT}^*(p_{YZ}^*\beta\cdot q_{ZT}^*\gamma)
\]
showing that $(\gamma\circ \beta)\circ\alpha$ is also equal to the push-forward along the $X\times Y\times Z\times T\to X\times T$ of the product of the pull-backs of $\alpha,\beta,\gamma$ along the respective projections.
\end{proof}

\subsection{Morphisms of schemes and finite MW-correspondences}\label{subsec:embedding}

Let $X,Y$ be smooth schemes of respective dimensions $d_X$ and $d_Y$. Let $f:X\to Y$ be a morphism and let $\Gamma_f:X\to X\times Y$ be its graph. Then $\Gamma_f(X)$ is of codimension $d_Y$ in $X\times Y$, finite and surjective over $X$. If $p_X:X\times Y\to X$ is the projection map, then $p_X\Gamma_f=\id$ and it follows that we have isomorphisms $\O_X\simeq \omega_{X/k}\otimes \Gamma_f^*p_X^*\omega_{X/k}^\vee$ and $\omega_{X\times Y/k}\otimes p_X^*\omega_{X/k}^\vee\simeq p_Y^*\omega_{Y/k}$, 
where the latter is obtained via the isomorphisms $\omega_{X\times Y/k}\simeq p_X^*\omega_{X/k}\otimes p_Y^*\omega_{Y/k}$ and $(-1)^{d_Xd_Y}$ the switch isomorphism $p_X^*\omega_{X/k}\otimes p_Y^*\omega_{Y/k}\simeq p_Y^*\omega_{Y/k}\otimes p_X^*\omega_{X/k}$.

Therefore we obtain a finite push-forward
\[
i_*: \sKMW_0(X)\to \chst {d_Y}{\Gamma_f}{X\times Y}{\omega_Y}
\]
We denote by $\graph f$ the class of $i_*(\langle 1\rangle)$ in $\chst {d_Y}{\Gamma_f}{X\times Y}{p^*\omega_Y}$. In particular, when $X=Y$ and $f=\id$, we set $1_X:=\graph {\id}$. Using \cite[Proposition 6.8]{Fasel07} we can check that $1_X$ is the identity for the composition defined in the previous section.

\begin{exem}\label{ex:action}
Let $X$ be a smooth scheme over $k$. The diagonal morphism induces a push-forward homomorphism
\[
\sKMW_0(X)\to \chst {d_X}{X}{X\times X}{\omega_X}
\] 
and a ring homomorphism $\sKMW_0(X)\to \cor k(X,X)$. For any smooth scheme $Y$, composition of morphisms endows the group $\cor k(Y,X)$ with the structure of a left $\sKMW_0(X)$-module and a right $\sKMW_0(Y)$-module.
\end{exem}

\begin{defin} \label{def:cortilde}
Let $\cor k$ be the category whose objects are smooth schemes and whose morphisms are the abelian groups $\cor k(X,Y)$ defined in Section \ref{sec:fGW}. We call it the \emph{category of finite MW-correspondences over $k$}.
\end{defin}

We see that $\cor k$ is an additive category, with disjoint union as direct sum.
We let the reader check that associating $\graph f$ to any morphism of smooth schemes $f:X\to Y$ gives a functor $\tilde\gamma:\sm{k}\to \cor k$.
\begin{rem} 
The category of finite correspondences as defined by Voevodsky can be recovered by replacing Chow-Witt groups by Chow groups in our definition. Indeed, when $Y$ is equidimensional of dimension $d=\dim Y$ and $T \in \Adm(X,Y)$,
\[
\CH^d_T(X \times Y) = \bigoplus_{x \in (X \times Y)^{(d)}\cap T} \Z
\]
since the previous group in the Gersten complex is zero because $T$ is $d$-dimensional, and the following group is also zero because there are no negative $\K$-groups. The composition of Voevodsky's finite correspondences coincides with ours as one can easily see from Lecture $1$ in \cite{Mazza06}. 

By the same procedure, it is of course possible to define finite correspondences using other cohomology theories with support, provided that they satisfy the classical axioms used in the definition of the composition (base change, etc.).
\end{rem}
The forgetful homomorphisms 
\[
\chst dT{X\times Y}{\omega_Y}\to \CH^d_T(X\times Y)
\] 
yield a functor $\pi:\cor k\to \ucor{k}$ (use \cite[Prop. 6.12]{Fasel07}) which is additive, and the classical functor $\gamma:\sm{k}\to \ucor{k}$ is the composite functor $\sm{k}\stackrel{\tilde\gamma}\to \cor k\stackrel{\pi}\to \ucor{k}$.

On the other hand, the hyperbolic homomorphisms
\[
\CH^d_T(X\times Y)\to \chst dT{X\times Y}{\omega_Y}
\]
yield a homomorphism $H_{X,Y}:\ucor{k}(X,Y)\to \cor k(X,Y)$ for any smooth schemes $X,Y$ (but not a functor $\ucor{k}\to \cor{k}$ since $H_{X,X}$ doesn't preserve the identity). The composite $\pi_{X,Y}H_{X,Y}$ is just the multiplication by $2$, as explained in Section \ref{sec:chowwitt}.

We now give two examples showing how to compose a finite MW-correspondence with a morphism of schemes.

\begin{exem}[Pull-back]\label{ex:flat_example}
Let $X,Y,U \in \sm k$ and let $f:X \to Y$ be a morphism. Let $(f\times 1):(X\times U)\to (Y\times U)$ be induced by $f$ and let $T\in \Adm(Y,U)$ be an admissible subset. Then $F:=(f\times 1)^{-1}(T)$ is an admissible subset of $X\times U$ by \cite[Lemma 1.6]{Mazza06}. It follows that the pull-back of cycles $(f\times 1)^*$ induces a homomorphism $\cor k(Y,U)\to \cor k(X,U)$. We let the reader check that it coincides with the composition with $\graph f$.
\end{exem}

\begin{exem}[Push-forwards]\label{ex:push-forwards}

Let $X$ and $Y$ be smooth schemes of dimension $d$ and let $f:X\to Y$ be a finite morphism such that any irreducible component of $X$ surjects to the irreducible component of $Y$ it maps to. Contrary to the classical situation, we don't have a finite MW-correspondence $Y\to X$ associated to $f$ in general, however, we can define one if $\omega_f$ admits an orientation.  

Let then $(\L,\psi)$ be an orientation of $\omega_f$. We define a finite MW-correspondence $\alpha(f,\L,\psi)\in \cor k(Y,X)$ as follows. Let $\Gamma^t_f:X\to Y\times X$ be the (transpose of the) graph of $f$. Then $X$ is an admissible subset and we have a transfer morphism 
\[
(\Gamma^t_f)_*:\sKMW_0(X,\omega_f)\to \chst {d}X{Y\times X}{\omega_X}.
\]
Composing with the homomorphism 
\[
\chst {d}X{Y\times X}{\omega_X}\to \cor k(Y,X),
\]
we get a map $\sKMW_0(X,\omega_f)\to \cor k(Y,X)$. Now the isomorphism $\psi$ together with the canonical isomorphism $\sKMW_0(X)\simeq \sKMW_0(X,\L\otimes\L)$ yield an isomorphism $\sKMW_0(X)\to \sKMW_0(X,\omega_f)$. We define the finite MW-correspondence $\alpha(f,\L,\psi)$ (or sometimes simply $\alpha(f,\psi)$) as the image of $\langle 1\rangle$ under the composite
\[
\sKMW_0(X)\to \sKMW_0(X,\omega_f)\to \cor k(Y,X).
\]
If $(\L^\prime,\psi^\prime)$ is equivalent to $(\L,\psi)$, then it is easy to check that the correspondences $\alpha(f,\L,\psi)$ and $\alpha(f,\L^\prime,\psi^\prime)$ are equal. Thus any element of $\Or(\omega_f)$ yields a finite MW-correspondence. In general, different choices of elements in $\Or(\omega_f)$ yield different correspondences.

When $g:Y \to Z$ is another such morphism with an orientation $(\mathcal M,\phi)$ of $\omega_g$, then $(\L \otimes f^*\mathcal M, \psi \otimes f^*\phi)$ is an orientation of $\omega_{g \circ f}=\omega_f \otimes f^*\omega_g$, and we have $\alpha(f,\L,\psi)\circ \alpha(g,\mathcal M,\phi) = \alpha(g \circ f,\L \otimes f^*\mathcal M, \psi \otimes f^*\phi)$. 

Let now $U$ be a smooth scheme of dimension $n$ and let $T\in \Adm(X,U)$. The commutative diagram
\[
\xymatrix{X\times U\ar[r]^{f\times 1}\ar[rd] & Y\times U\ar[d]^{p_Y} \\ 
 & Y}
\]
where $p_Y:Y\times U$ is the projection on the first factor and \cite[Lemma 1.4]{Mazza06} show that $(f\times 1)(T)\in A(Y,U)$ in our situation. Moreover, we have a push-forward morphism
\[
(f\times 1)_*:\chst nT{X\times U}{\omega_U\otimes \omega_f}\to \chst n{(f\times 1)(T)}{Y\times U}{\omega_U}
\] 
Using the trivialization $\psi$, we get a push-forward morphism
\[
(f\times 1)_*:\chst nT{X\times U}{\omega_U}\to \chst n{(f\times 1)(T)}{Y\times U}{\omega_U}
\] 
Now this map commutes to the extension of support homomorphisms, and it follows that we get a homomorphism 
\[
(f\times 1)_*:\cor k(X,U)\to \cor k(Y,U)
\]
depending on $\psi$. We let the reader check that $(f\times 1)_*(\beta)=\beta\circ \alpha(f,\psi)$ for any $\beta\in \cor k(X,U)$, using the base change formula as well as the projection formula.

In particular, when $U=Y$ and $\beta=\graph f$, using $\psi$ we can push-forward along $f$ as $\sKMW_0(X) \simeq \sKMW_0(X,\omega_f) \to \sKMW_0(Y)$ to obtain an element $f_*\langle 1 \rangle$ in $\sKMW_0(Y)$, and we have $\graph f \circ \alpha(f,\psi)=f_*\langle 1 \rangle \cdot \id_Y$, using the action of Example \ref{ex:action}. 
\end{exem}

\begin{rem}\label{rem:notsurjective}
Suppose that $X$ is connected, $f:X\to Y$ is a finite surjective morphism with relative bundle $\omega_f$ and that $\omega_f\neq 0$ in $\Pic(X)/2$ (take for instance the map $\mathbb{P}^2\to \mathbb{P}^2$ defined by $[x_0:x_1:x_2]\mapsto [x_0^2:x_1^2:x_2^2]$). Consider the finite correspondence $Y\to X$ corresponding to (the transpose of) the graph of $f$. As in the previous example, we have a transfer homomorphism 
\[
(\Gamma_f)_*:\sKMW_0(X,\omega_f)\to \chst {d}X{Y\times X}{\omega_X}.
\]
making the diagram
\[
\xymatrix{\sKMW_0(X,\omega_f)\ar[r]^-{\simeq} \ar[d] & \chst {d}X{Y\times X}{\omega_X}\ar[d] \\
\Z\ar[r]^-{\simeq} & \CH^d_X(Y\times X)}
\]
commutative, where the vertical homomorphisms are the forgetful maps and the horizontal ones are isomorphisms, as one clearly sees using the Milnor-Witt Gersten complex. Since $\omega_f$ is not a square in $\Pic(X)$, the left-hand vertical map is not surjective: it is equal to the rank map. It follows that the map $\cor k(Y,X)\to \ucor{k}(Y,X)$ is not surjective. Indeed, if we enlarge the support to a larger admissible set $X'$, one of its irreducible components will be $X$, and the (usual) Chow groups with support in $X'$ will be isomorphic to several copies of $\Z$, one for each irreducible component, and the forgetful map cannot surject to the copy of $\Z$ corresponding to $X$.
\end{rem}

\subsection{Tensor products}

Let $X_1,X_2,Y_1,Y_2$ be smooth schemes over $\spec k$. Let $d_1=\dim Y_1$ and $d_2=\dim Y_2$.

Let $\alpha_1\in \chst {d_1}{T_1}{X_1\times Y_1}{\omega_{Y_1}}$ and $\alpha_2\in \chst {d_2}{T_2}{X_2\times Y_2}{\omega_{Y_2}}$ for some admissible subsets $T_i\subset X_i\times Y_i$. The exterior product defined in \cite[\S 4]{Fasel07} gives a cycle 
\[
(\alpha_1\times \alpha_2)\in \chst {d_1+d_2}{T_1\times T_2}{X_1\times Y_1\times X_2\times Y_2}{p_{Y_1}^*\omega_{Y_1/k}\otimes p_{Y_2}^*\omega_{Y_2/k}}
\]
where $p_{Y_i}:X_1\times Y_1\times X_2\times Y_2\to Y_i$ is the projection to the corresponding factor. 
Let $\sigma:X_1\times Y_1\times X_2\times Y_2\to X_1\times X_2\times Y_1\times Y_2$ be the transpose isomorphism. Applying $\sigma_*$, we get a cycle
\[
\sigma_*(\alpha_1\times \alpha_2)\in \chst {d_1+d_2}{\sigma(T_1\times T_2)}{X_1\times X_2\times Y_1\times Y_2}{p_{Y_1}^*\omega_{Y_1/k}\otimes p_{Y_2}^*\omega_{Y_2/k}}.
\]
On the other hand, $p_{Y_1}^*\omega_{Y_1/k}\otimes p_{Y_2}^*\omega_{Y_2/k}=\omega_{Y_1\times Y_2}$ and it is straightforward to check that (the underlying reduced scheme) $\sigma(T_1\times T_2)$ is finite and surjective over $X_1\times X_2$. Thus $\sigma_*(\alpha_1\times \alpha_2)$ defines a finite MW-correspondence between $X_1\times X_2$ and $Y_1\times Y_2$. 

\begin{defin}\label{def:tensor}
If $X_1$ and $X_2$ are smooth schemes over $k$, we define their tensor product as $X_1\otimes X_2:=X_1\times X_2$. If $\alpha_1\in \cor k(X_1,Y_1)$ and $\alpha_2\in \cor k(X_2,Y_2)$, then we define their tensor product as $\alpha_1\otimes \alpha_2:=\sigma_*(\alpha_1\times \alpha_2)$. 
\end{defin}

\begin{lem}
The tensor product $\otimes$, together with the obvious associativity and symmetry isomorphisms endows $\cor k$ with the structure of a symmetric monoidal category. 
\end{lem}

\begin{proof}
Straightforward.
\end{proof}


\section{Presheaves on $\cor k$}

\begin{defin}
A \emph{presheaf with MW-transfers} is a contravariant additive functor $\cor k\to {\mathcal Ab}$. We will denote by $\psh k$ the category of presheaves with MW-transfers. Let $\tau$ be $\Zar$, $\Nis$ and $\Et$, respectively the Zariksi, Nisnevich or \'etale topology on $\sm k$. We say that a presheaf with MW-transfers is a $\tau$-sheaf with MW-transfers if when restricted to $\sm k$, it is a sheaf in the $\tau$-topology. We denote by $\sh k{\tau}$ the category of $\tau$-sheaves with MW-transfers.
\end{defin}

\begin{rem}
The sheaves with MW-transfers are closely related to sheaves with generalized transfers as defined in \cite[Definition 5.7]{Morel11}. Indeed, let $M$ be a Nisnevich sheaf with MW-transfers. Then it is endowed with an action of $\sKMW_0$ by Example \ref{ex:action}. Following the procedure described in Section \ref{sec:limits}, one can define $M(F)$ for any finitely generated field extension $F/k$. If $F\subset L$ is a finite field extension, then the canonical orientation of Lemma \ref{lem:orientation} together with the push-forwards defined in Example \ref{ex:push-forwards} show that we have a homomorphism $\Tr_F^L:M(L)\to M(F)$. One can then check that the axioms listed in \cite[Definition 5.7]{Morel11} are satisfied. Conversely, we don't know if a Nisnevich sheaf with generalized transfers in the sense of Morel yields a Nisnevich sheaf with MW-transfers but this seems quite plausible. 
\end{rem}

Recall first that there is a forgetful additive functor $\pi:\cor k\to \ucor{k}$. It follows that any (additive) presheaf on $\ucor{k}$ defines a presheaf on $\cor k$ by composition. We now give a more exotic example. 

\begin{lem}
For any $j\in\Z$, the contravariant functor $X\mapsto \sKMW_j(X)$ is a presheaf on $\cor k$.
\end{lem}

\begin{proof}
Let then $X,Y$ be smooth schemes and $T\subset X\times Y$ be an admissible subset. Let $\beta \in \sKMW_j(Y)$ and $\alpha\in \chst {d_Y}T{X\times Y}{\omega_Y}$ be cycles. We set 
\[
\alpha^*(\beta):= p_*(p_Y^*(\beta)\cdot \alpha)
\]
where $p$ and $p_Y$ are the respective projections, and $p_*$ is defined using the canonical isomorphism $\omega_Y\simeq \omega_{X\times Y}\otimes \omega_X^\vee$. We let the reader check that $\alpha^*$ is additive. Next, we observe that $p_Y^*(\beta)$ commutes with any element in the total Chow-Witt group of $X\times Y$ (\cite[Lemmas 4.19 and 4.20]{Fasel07} or \cite{Calmes17}).

If $T\subset T^\prime$, we have a commutative diagram
\[
\xymatrix{\chst {d_Y}T{X\times Y}{\omega_Y}\ar[r]\ar[rd]_{p_*} & \chst {d_Y}{T^\prime}{X\times Y}{\omega_Y}\ar[d]^-{p_*} \\
 & \ch iX}
\]
where the top horizontal morphism is the extension of support. It follows that $\alpha\mapsto \alpha^*$ defines a map $\cor k(X,Y)\to \homm {{\mathcal A}b}{\H^i(Y,\sKMW_j)}{\H^i(X,\sKMW_j)}$. We now check that this map preserves the respective compositions. With this in mind, consider again the diagram (\ref{eqn:composition})
\[
\begin{gathered}
\xymatrix{
X\times Z \ar@/^2em/[rrrd]^-{r_Z}\ar@/_1.3em/[rddd]_-{p_X} & & & \\
 & X\times Y\times Z\ar[r]^-{q_{YZ}}\ar[d]_-{p_{XY}}\ar[lu]_-{q_{XZ}} & Y\times Z\ar[r]^-{q_Z}\ar[d]^-{p_Y} & Z \\
 & X\times Y\ar[r]_-{q_Y}\ar[d]_-p & Y & \\
 & X. & & 
}
\end{gathered}
\]
Let $\alpha_1\in \chst {d_Y}{T_1}{X\times Y}{\omega_Y}$ and $\alpha_2\in \chst {d_Z}{T_2}{Y\times Z}{\omega_Z}$ be correspondences, with $T_1\subset X\times Y$ and $T_2\subset Y\times Z$ admissible. Let moreover $\beta\in \sKMW_j(Z)$. By definition, we have 
\[
(\alpha_2\circ\alpha_1)^*(\beta)=(p_X)_*[r_Z^*\beta\cdot (q_{XZ})_*(p_{XY}^*\alpha_1\cdot q_{YZ}^*\alpha_2)].
\]
Using the projection formula and the fact that $r_Z^*\beta$ commutes with the total Chow-Witt group, we have 
\begin{eqnarray*}
 (p_X)_*[r_Z^*\beta\cdot (q_{XZ})_*(p_{XY}^*\alpha_1\cdot q_{YZ}^*\alpha_2)]& = & (p_X)_*[(q_{XZ})_*(p_{XY}^*\alpha_1\cdot q_{YZ}^*\alpha_2)\cdot r_Z^*\beta]   \\
 & = & (p_X)_*[(q_{XZ})_*(p_{XY}^*\alpha_1\cdot q_{YZ}^*\alpha_2\cdot q_{XZ}^*r_Z^*\beta)] \\
 & = & p_*(p_{XY})_*(p_{XY}^*\alpha_1\cdot q_{YZ}^*\alpha_2\cdot q_{XZ}^*r_Z^*\beta).
\end{eqnarray*} 
On the other hand,
\begin{eqnarray*}
\alpha_1^*\circ \alpha_2^*(\beta) & = & \alpha_1^*((p_Y)_*(q_Z^*\beta\cdot \alpha_2)) \\
 & = & p_*(q_Y^*(p_Y)_*(q_Z^*\beta\cdot \alpha_2)\cdot \alpha_1).
\end{eqnarray*}
By base change, $q_Y^*(p_Y)_*=(p_{XY})_*q_{YZ}^*$ and it follows that 
\begin{eqnarray*}
\alpha_1^*\circ \alpha_2^*(\beta) & = & p_*((p_{XY})_*(q_{YZ}^*q_Z^*\beta\cdot q_{YZ}^*\alpha_2)\cdot \alpha_1) \\
 & =  & 
\end{eqnarray*}
Using the projection formula once again, the latter is equal to 
\[
p_*(p_{XY})_*(q_{YZ}^*q_Z^*\beta\cdot q_{YZ}^*\alpha_2\cdot p_{XY}^*\alpha_1)
\]
We conclude using Remark \ref{rem:everythingcommutes} and the fact that $\beta$ commutes with the total Chow-Witt group.
\end{proof}

\begin{rem}
More generally, the contravariant functor $X\mapsto \H^i(X,\sKMW_j)$ is a presheaf on $\cor k$ for any $j\in\Z$ and $i\in\N$. The above proof applies with some modifications (taking into account that $\beta$ now doesn't commute with the total Chow-Witt group), or we can alternatively use the facts that $\sKMW_j$ is a homotopy invariant Nisnevich sheaf with MW-transfers, that the category of Nisnevich sheaves with MW transfers has enough injectives and that Zariski and Nisnevich cohomology coincide for invariant Nisnevich sheaves with MW-transfers. We refer the reader to \cite{Deglise16} for details.

In contrast, the presheaf $X\mapsto \H^i(X,\sKMW_j)$ doesn't have transfers in the sense of Voevodsky, as the projective bundle formula doesn't hold in that setting (see \cite[Theorem 11.7]{Fasel09d}).
\end{rem}

\subsection{Extending presheaves to limits}\label{sec:limits}

We consider the category $\cP$ of filtered projective systems $((X_\lambda)_{\lambda \in I},f_{\lambda\mu})$ of smooth quasi-compact schemes over $k$, with affine \'etale 
transition morphisms $f_{\lambda\mu}:X_\lambda \to X_\mu$. Morphisms in $\cP$ are defined by $\Hom((X_\lambda)_{\lambda \in I}, (X_\mu)_{\mu \in J})=\varprojlim \big(\varinjlim \Hom(X_\lambda,X_{\mu})\big)$, as in \cite[8.13.3]{EGAIV3}. The limit of such a system exists in the category of schemes by loc.\ cit.\ 8.2.3, and by 8.13.5, the functor sending a projective system to its limit defines an equivalence of categories from $\cP$ to the full subcategory $\barsm{k}$ of schemes over $k$ that are limits of such systems.
\begin{rem}
It follows from this equivalence of categories that such a projective system converging to a scheme that is already finitely generated (e.g.\ smooth) over $k$ has to be constant above a large enough index.
\end{rem}

Let now $F$ be a presheaf of abelian groups on $\sm{k}$. We extend $F$ to a presheaf $\bar F$ on $\cP$ by setting on objects $\bar F((X_\lambda)_{\lambda \in I})=\varinjlim_{\lambda\in I} F(X_\lambda)$. An element of the set $\varprojlim \big(\varinjlim \Hom(X_\lambda,X_{\mu})\big)$ yields a morphism $\varinjlim_{\mu\in J} F(X_\mu)\to \varinjlim_{\lambda\in I} F(X_\lambda)$, this respects composition, and thus $\bar F$ is well-defined. Using the above equivalence of categories, it follows immediately that $\bar F$ defines a presheaf on $\barsm{k}$ which extends $F$ in the sense that $\bar F$ and $F$ coincide on $\sm{k}$ since a smooth scheme can be considered as a constant projective system.

Since any finitely generated field extension $L/k$ can be written as a limit of smooth schemes in the above sense, we can consider in particular $F(\spec L)$ and to shorten the notation, we often write $F(L)$ instead of $\bar F(\spec L)$ in what follows.

We will mainly apply this limit construction when $F$ is $\cor k(- \times X,Y)$, $\sKMW_*$ or the MW-motivic cohomology groups $\HMW^{p,q}(-,\Z)$, to be defined in \ref{def:genmotiviccohom}. 
\medskip

We now slightly extend this equivalence of categories to a framework useful for Chow-Witt groups with support. We consider the category $\cT$ of triples $(X,Z,\L)$ where $X$ is scheme of finite type over $k$, with a closed subset $Z$ and a line bundle $\L$ over $X$. A morphism from $(X_1,Z_1,\L_1)$ to $(X_2,Z_2,L_2)$ in this category is a pair $(f,i)$ where $f:X_1 \to X_2$ is a morphism of $k$-schemes such that $f^{-1}(Z_2)\subseteq Z_1$ and $i:f^*\L_2 \to \L_1$ is an isomorphism of line bundles over $X_1$. The composition of two such morphisms $(f,i)$ and $(g,j)$ is defined as $(f\circ g, j \circ g^*(i))$. Let $\tP$ be the category of projective systems in that category such that:
\begin{itemize}
\item the objects are regular and the transition maps are affine \'etale;
\item beyond any index, there is a $\mu$ such that for any $\lambda$ beyond $\mu$, we have $f_{\lambda\mu}^{-1}(Z_\mu)=Z_\lambda$,
\end{itemize}
with morphisms defined by a double limit of $\Hom$ groups in $\cT$, as above.
Let $X$ be the inverse limit of the $X_\lambda$. All the pull-backs of the various $\L_\lambda$ to $X$ are canonically isomorphic by pulling back the isomorphisms $i_{\lambda\mu}$, and by the last condition, the closed subsets $Z_\lambda$ stabilize to a closed subset $Z$ of $X$. In other words, the inverse limit of the system exists in $\cT$. 

\begin{prop} \label{prop:equivtriple}
The functor sending a system to its inverse limit is an equivalence of categories from $\tP$ to the full subcategory of $\cT$ of objets $(X,Z,\L)$ such that $X$ is a projective limit of regular schemes with affine \'etale transition morphisms. 
\end{prop}
\begin{proof}
It follows from \cite{EGAIV3}, 8.13.5 for the underlying schemes, 8.3.11 for the closed subset (since our schemes are of finite type over a field, they are Noetherian and every closed subset is constructible), 8.5.2 and 8.5.5 for the line bundle.
\end{proof}

As previously, any presheaf of abelian groups or sets $F$ defined on the full subcategory of $\cT$ with regular underlying schemes can be extended uniquely to the subcategory of $\cT$ with underlying schemes that are limits. 

If the limit scheme $X$ happens to be regular, for any nonnegative integer $i$, the pull-back induces a map
\[
\varinjlim_{\lambda \in I} \chst i{Z_\lambda}{X_\lambda} {\L_\lambda} \to \chst i Z X \L.
\]
\begin{lem} \label{lem:limitChowiso}
This map is an isomorphism.
\end{lem}
\begin{proof}
Since $(X,Z,\L)$ can be considered as a constant projective system, it follows from the equivalence of categories of Proposition \ref{prop:equivtriple}.
\end{proof}

\subsection{Representable presheaves} \label{sec:representable}

\begin{defin}
Let $X \in \sm k$. We denote by $\tZX X$ the presheaf $\cor k(-,X)$. If $x:\spec k \to X$ is a rational point, we denote by $\tZX {X,x}$ the cokernel of the morphism $\tZX {\spec k} \to \tZX X$ (which is split injective). More generally, let $X_1,\ldots,X_n$ be smooth schemes pointed respectively by the rational points $x_1,\ldots,x_n$. We define $\tZX {(X_1,x_1)\wedge\ldots\wedge (X_n,x_n)}$ as the cokernel of the split injective map
\[
\xymatrix@C=20ex{
\displaystyle \bigoplus_i \tZX {X_1\times \ldots \times X_{i-1}\times X_{i+1}\times \ldots \times X_n}\ar[r]^-{\id_{X_1}\times \cdots \times x_i \times \cdots \times \id_{X_n}} &  \tZX {X_1\times\ldots\times X_n}.
}
\]
\end{defin}

\begin{exem}
It follows from Example \ref{ex:basic} that $\tZX {\spec k}=\sKMW_0$, and we write it $\tZ(k)$ or even $\tZ$ when there is no possible confusion on the base field.
\end{exem}

Our next goal is to check that for any smooth scheme $X$, the presheaf $\tZX X$ is actually a sheaf in the Zariski topology. We start with an easy lemma.

\begin{lem}\label{lem:unramified}
Let $Y \in \sm k$ be connected, with function field $k(Y)$. Then for any $X \in \sm k$, the homomorphism $\cor k(Y,X)\to \cor k(k(Y),X)$ is injective.
\end{lem}

\begin{proof}
It follows from Lemma \ref{lem:restrictioninj} that all transition maps in the system converging to $\cor k(k(Y),X)$ are injective. 
\end{proof}

\begin{prop}
For any $X \in \sm k$, the presheaf $\tZX X$ is a sheaf in the Zariski topology.
\end{prop}

\begin{proof}
It suffices to prove that for any cover of a scheme $Y$ by two Zariski open sets $U$ and $V$, the equalizer sequence
\[
\xymatrix@1{
\cor k(Y,X) \ar[r] & \cor k(U,X) \coprod \cor k(V,X) \ar@<-.5ex>[r] \ar@<.5ex>[r] & \cor k(U \cap V,X)
}
\]
is exact in the well-known sense. Injectivity of the map on the left follows from Lemma \ref{lem:unramified}. To prove exactness in the middle, let $\alpha$ and $\beta$ be elements in the middle groups, equalizing the right arrows and with respective supports $E \in \Adm(U,X)$ and $F \in \Adm(V,X)$ by Lemma \ref{lem:supportadmis}. Since $\alpha$ restricted to $V$ is $\beta$ restricted to $U$, we must have $E \cap V=F \cap U$ by Lemma \ref{lem:supporttoopen}. By Lemma \ref{lem:admissiblesheaf}, the closed set $T=E \cup F$ is admissible, and the conclusion now follows from the long exact sequence of localization for Chow-Witt groups with support in $T$. 
\end{proof}

\begin{exem}
In general, the Zariski sheaf $\tZX X=\cor k(-,X)$ is not a Nisnevich sheaf (and therefore not an \'etale sheaf either). Set $\A^1_{a_1,\ldots,a_n}:=\A^1\setminus \{a_1,\ldots,a_n\}$ for any rational points $a_1,\ldots,a_n\in \A^1(k)$ and consider the Nisnevich square
\[
\xymatrix{\A^1_{0,1,-1}\ar[r]\ar[d]_-g & \A^1_{0,-1}\ar[d]^-f \\
\A^1_{0,1}\ar[r] & \A^1_0}
\]
where the morphism $f:\A^1_{0,-1}\to \A^1_0$ is given by $x\mapsto x^2$, the horizontal maps are inclusions and $g$ is the base change of $f$. 

We now show that $\tZX {\A^1_{0,1}}$ is not a Nisnevich sheaf. In order to see this, we prove that the sequence
\[
\tZX {\A^1_{0,1}}(\A^1_{0})\to \tZX {\A^1_{0,1}}(\A^1_{0,1})\oplus \tZX {\A^1_{0,1}}(\A^1_{0,-1})\to \tZX {\A^1_{0,1}}(\A^1_{0,1,-1})
\]
is not exact in the middle. Let $\Delta:\A^1_{0,1}\to \A^1_{0,1} \times \A^1_{0,1}$ be the diagonal embedding. It induces an isomorphism
\[
\sKMW_0(\A^1_{0,1})\to \chst 1{\Delta(\A^1_{0,1})}{\A^1_{0,1}\times \A^1_{0,1}}{\omega_{\A^1_{0,1}}}
\]
and thus a monomorphism $\sKMW_0(\A^1_{0,1})\to \tZX {\A^1_{0,1}}(\A^1_{0,1})$. Since the restriction homomorphism $\sKMW_0(\A^1_0)\to \sKMW_0(\A^1_{0,1})$ is injective, it follows that the class $\eta\cdot [t]=-1+\langle t\rangle $ is non trivial in $\sKMW_0(\A^1_{0,1})$ and its image $\alpha_t$ in $\tZX {\A^1_{0,1}}(\A^1_{0,1})$ is also non trivial. We claim that its restriction in $\tZX {\A^1_{0,1}}(\A^1_{0,1,-1})$ is trivial. Indeed, consider the Cartesian square
\[
\xymatrix{\A^1_{0,1,-1} \ar[r]^-{\Gamma_g}\ar[d]_-g & \A^1_{0,1,-1}\times \A^1_{0,1}\ar[d]^-{(g\times 1)} \\
\A^1_{0,1}\ar[r]_-{\Delta} & \A^1_{0,1}\times \A^1_{0,1}}
\]
From Example \ref{ex:flat_example}, we know that the restriction of $\alpha_t$ to $\tZX {\A^1_{0,1}}(\A^1_{0,1,-1})$ is represented by $(g\times 1)\Delta_*(-1+\langle t\rangle)$. By base change, the latter is $(\Gamma_g)_*g^*(\alpha_t)$. Now we have $g^*(-1+\langle t\rangle)=-1+\langle t^2\rangle=0$ in $\sKMW_0(\A^1_{0,1,-1})$. In conclusion, we see that the correspondence $(\alpha_t,0)$ is in the kernel of the homomorphism
\[
\tZX {\A^1_{0,1}}(\A^1_{0,1})\oplus \tZX {\A^1_{0,1}}(\A^1_{0,-1})\to \tZX {\A^1_{0,1}}(\A^1_{0,1,-1}).
\] 
To conclude that $\tZX {\A^1_{0,1}}$ is not a sheaf, it then suffices to show that $\alpha_t$ cannot be the restriction of a correspondence in $\tZX {\A^1_{0,1}} (\A^1_0)$. To see this, observe first that the restriction map $\Adm(\A^1_0,\A^1_{0,1})\to \Adm(\A^1_{0,1},\A^1_{0,1})$ is injective as well as the homomorphism $\tZX {\A^1_{0,1}}(\A^1_0)\to \tZX {\A^1_{0,1}}(\A^1_{0,1})$. Suppose then that $\beta\in \tZX {\A^1_{0,1}}(\A^1_0)$ is a correspondence whose restriction is $\alpha_t$, and let $T=\mathrm{supp}(\beta)$. From the above observations, we see that $\Delta(\A^1_{0,1})=\mathrm{supp}(\alpha_t)=T\cap (\A^1_{0,1}\times \A^1_{0,1})$. In particular, we see that $T$ is irreducible and its generic point is the generic point of the (transpose of the) graph of the inclusion $\A^1_{0,1}\to \A^1_0$. But the graph is closed, but not finite over $\A^1_0$. It follows that $\beta$ doesn't exist and thus that $\tZX {\A^1_{0,1}}$ is not a Nisnevich sheaf.
\end{exem}

Next, recall that the presheaf $\ucor k(\_,X):=\Ztr X$ is also a presheaf with MW-transfers (using the functor $\cor k\to \ucor k$). The morphism $\tZX X\to \Ztr X$ is easily seen to be a morphism of MW-presheaves. 

\begin{lem}
The morphism of MW-presheaves $\tZX X\to \Ztr X$ induces an epimorphism of Zariski sheaves.
\end{lem}

\begin{proof}
Let $Y$ be the localization of a smooth scheme at a point. We have to show that the map $\tZX X(Y)\to \Ztr X(Y)$ is surjective. It is sufficient to prove that for any elementary correspondence $T\subset Y\times X$, the map
\[
\chst {d_X}T{Y\times X}{\omega_X}\to \CH^{d_X}_T(Y\times X)=\Z
\]
is surjective. Note that the map $T\to Y\times X\to Y$ is finite and surjective, and therefore that $T$ is the spectrum of a semi-local domain. Let $\tilde T$ be the normalization of $T$ into its field of fractions. Note that the singular locus of $\tilde T$ is of codimension at least $2$ and therefore that its image in $Y\times X$ is of codimension at least $d_X+2$. As both $\chst {d_X}T{Y\times X}{\omega_X}$ and $\CH^{d_X}_T(Y\times X)$ are invariant if we remove points of codimension $d_X+2$, we may suppose that $\tilde T$ is smooth and therefore that we have a well-defined push-forward
\[
i_*:\cht 0{\tilde T}{\omega_{{\tilde T}/k}\otimes i^*L}\to \chst {d_X}T{Y\times X}{\omega_{Y\times X/k}\otimes L} 
\] 
for any line bundle $L$ on $Y\times X$, where $i$ is the composite $\tilde T\to T\subset Y\times X$. As $\tilde T$ is semi-local, we can find a trivialization of $\omega_{{\tilde T}/k}\otimes i^*L$ and get a push-forward homomorphism
\[
i_*:\ch 0{\tilde T}\to \chst {d_X}T{Y\times X}{\omega_{Y\times X/k}\otimes L} 
\] 
Now, $\langle 1\rangle\in \ch 0{\tilde T}$ and we can use the commutative diagram
\[
\xymatrix{\ch 0{\tilde T}\ar[r]^-{i_*}\ar[d] & \chst {d_X}T{Y\times X}{\omega_{Y\times X/k}\otimes L}\ar[d] \\
\CH^0(\tilde T)\ar[r]_-{i_*} & \CH^{d_X}_T(Y\times X)}
\]
to conclude that the composite 
\[
\ch 0{\tilde T}\to \chst {d_X}T{Y\times X}{\omega_{Y\times X/k}\otimes L}\to \CH^{d_X}_T(Y\times X)
\]
is surjective for any line bundle $L$. The claim follows.
\end{proof}

\begin{rem}
The same proof works in both the Nisnevich and the \'etale topologies.
\end{rem}

Now, we turn to the task of determining the kernel of the morphism $\tZX X\to \Ztr X$. With this in mind, recall that we have for any $n\in \N$ an exact sequence of sheaves
\[
0\to \mathrm{I}^{n+1}\to \KMW_n \to \KM_n \to 0.
\] 
If $T\subset Y\times X$ is of pure codimension $d_X$, the long exact sequence associated to the previous exact sequence reads as
\[
\H^{d_X}_T(Y\times X,\mathrm{I}^{d_X+1},\omega_X)\to \chst {d_X}T{Y\times X}{\omega_X} \to \CH^{d_X}_T(Y\times X) \to \ldots
\]
and the left-hand morphism is injective since there are no points of codimension $d_X-1$ supported on $T$. Moreover, if $\alpha\in \cor k(Y^\prime,Y)$ and $\pi(\alpha)\in \ucor k(Y^\prime,Y)$ is its image in $\ucor k$ under the usual functor, the diagram
\[
\xymatrix{\chst {d_X}T{Y\times X}{\omega_X} \ar[r]\ar[d]_-{\alpha^*} & \CH^{d_X}_T(Y\times X)\ar[d]^-{\pi(\alpha)^*} \\
\chst {d_X}T{Y^\prime\times X}{\omega_X} \ar[r] & \CH^{d_X}_T(Y^\prime\times X)}
\]
commutes, where the vertical arrows are the respective composition with $\alpha$ and $\pi(\alpha)$. These observations motivate the following definition.

\begin{defin}
For any smooth connected schemes $X,Y$, we set 
\[
\icor k(Y,X)=\lim_{T\in \Adm(Y,X)} \H^{d_X}_T(Y\times X,\mathrm{I}^{d_X+1},\omega_X).
\]
As before, we extend this definition to any smooth scheme $X,Y$ by additivity. The above diagram shows that $\icor k(\_,X)$ is an element of $\psh k$ that we denote by $\itZX X$.
\end{defin}

The above arguments prove that we have an exact sequence of Zariski sheaves
\[
0\to \itZX X \to \tZX X\to \Ztr X\to 0.
\]

\begin{rem}
Instead of $\icor k(Y,X)$, we could consider the abelian group
\[
\lim_{T\in \Adm(Y,X)} \H^{d_X}_T(Y\times X,\mathrm{I}^{d_X},\omega_X)=\lim_{T\in \Adm(Y,X)} \H^{d_X}_T(Y\times X,\mathrm{W},\omega_X)
\]
In this fashion, we would obtain a category of sheaves with Witt-transfers. We postpone the study of this category to further work.
\end{rem}

The category $\psh k$ of presheaves with MW-transfers admits a unique symmetric monoidal structure such that the embedding $\cor k\to  \psh k$ given by $X\mapsto \tZX X$ is symmetric monoidal (\cite[\S 1.2]{Deglise16}). We denote by $F\otimes G$ the tensor product of two presheaves, and we observe that $\tZX X\otimes \tZX Y=\tZX {X\times Y}$ by definition. It turns out that $\psh k$ is a closed monoidal category, with internal Hom functor $\uHom$ determined by $\uHom(\tZX X,F)=F(X\times \_)$. The proof of the next result is easy and we thus omit it.

\begin{lem}
Suppose that $F\in \psh k$ is a $\tau$-sheaf. Then $\uHom(\tZX X,F)$ is also a $\tau$-sheaf for any $X \in \sm k$.
\end{lem}

\begin{defin}
Let $F\in\psh k$ and $n \in \N$. Then the $n$-th contraction of $F$, denoted by $F_{-n}$, is the presheaf $\uHom(\tZX {(\gm,1)^{\wedge n}},F)$. Thus, $F_{-n}=(F_{-n+1})_{-1}$.
\end{defin}

\begin{exem}\label{ex:MilnorWittcontraction}
If $F=\sKMW_j$ for some $j\in\Z$, then it follows from Lemma \ref{lem:explicitcontraction} that $F_{-1}=\sKMW_{j-1}$.
\end{exem}

\begin{defin}
A presheaf $F$ in $\psh k$ is said to be homotopy invariant if the map
\[
F(X)\to F(X\times\A^1)
\]
induced by the projection $X\times \A^1\to X$ is an isomorphism for any $X\in \sm{k}$.
\end{defin}

\begin{exem}
We already know that $\tZX k$ coincides with the Nisnevich sheaf $\sKMW_0$. It follows from \cite[Corollaire 11.3.3]{Fasel08a} that this sheaf is homotopy invariant.
\end{exem}

\subsection{The module structure}

Recall from Example \ref{ex:action} that the Zariski sheaf $\tZX X$ is endowed with an action of a left $\sKMW_0(X)$-module for any smooth scheme $X$. Pulling back along the structural morphism $X\to \spec k$, we obtain a ring homomorphism $\KMW_0(k)\to \sKMW_0(X)$ and it follows that $\tZX X$ is a sheaf of $\KMW_0(k)$-modules. If $X\to Y$ is a morphism of smooth schemes, it is readily verified that the morphism $\tZX X\to \tZX Y$ is a morphism of sheaves of $\KMW_0(k)$-modules. 

Let now $F$ be an object of $\psh k$. If $X$ is a smooth scheme, the presheaf $\uHom (\tZX X,F)$ is naturally endowed with the structure of a right $\KMW_0(k)$-module by precomposing with the morphism $\tZX X\to \tZX X$ induced by some $\alpha\in \KMW_0(k)$. As $\KMW_0(k)$ is commutative, it follows that $\uHom (\tZX X,F)$ is also a sheaf of (left) $\KMW_0(k)$-modules. If $X\to Y$ is a morphism of smooth schemes, a direct computation shows that the induced morphism $\uHom (\tZX Y,F)\to \uHom (\tZX X,F)$ is a morphism of $\KMW_0(k)$-modules.

\section{Motivic cohomology}

We define the pointed scheme $\gmpt =(\gm,1)$.

\begin{defin}
For any $q\in\Z$, we define the Zariski sheaf $\tZ\{q\}$ by
\[
\tZ\{q\}:=
\begin{cases} \tZX {\gmpt^{\wedge q}} & \text{if $q>0$.} \\
\tZX k& \text{if $q=0$}. \\
\uHom (\gmpt^{\wedge q}, \tZX k) & \text{if $q<0$.}
\end{cases}
\]
\end{defin}
Note that the sheaves considered are all sheaves of $\KMW_0(k)$-modules and that for $q<0$,the sheaf $\tZ\{q\}$ is the $(-q)$-th contraction of $\tZX k=\KMW_0$. It follows thus from Example \ref{ex:MilnorWittcontraction} that $(\sKMW_0)_{q}=\sKMW_q=\mathbf{W}$ for $q<0$, where the latter is the Zariski sheaf associated to the presheaf $X\mapsto \W(X)$ (Witt group).

\begin{rem}
Of course, we could have defined $\tZ\{q\}$ to be $\mathbf{W}$ when $q<0$ from the beginning. The advantage of our definition is to make the product structure defined below more explicit.  
\end{rem}

\subsection{Motivic sheaves}\label{section:motivic}

Let $\Delta^{\bullet}$ be the cosimplicial object on $\sm{k}$ defined by 
\[
\Delta^n:=\spec {k[x_0,\ldots,x_n]/(\sum_{i=0}^nx_i-1)}.
\]
and the usual faces and degeneracy maps. Given any $F\in \psh k$, we get a simplicial object $\uHom(\tZX {\Delta^{\bullet}},F)$ in $\psh k$ which is a simplicial sheaf in case $F$ is one. 

\begin{defin}
The \emph{Suslin-Voevodsky singular construction on $F$} is the complex associated to the simplicial object $\uHom(\tZX {\Delta^{\bullet}},F)$. Following the conventions, we denote it by $\Cstar F$.
\end{defin}

The following lemma is well-known and we let the proof to the reader.

\begin{lem}
Suppose that $F$ is a homotopy invariant presheaf on $\cor k$. Then the natural homomorphism $F\to \Cstar F$ is a quasi-isomorphism of complexes (here $F$ is considered as a complex concentrated in degree $0$).
\end{lem}

\begin{defin}
For any integer $q\in\Z$, we define $\tZ(q)$ as the complex of Zariski sheaves of $\KMW_0(k)$-modules $\Cstar \tZ\{q\}[-q]$.
\end{defin}

As opposed to \cite{Mazza06}, we use the homological notation and we then have by convention $\tZ(q)_i=\Ci {q+i}\tZ\{q\}$. It follows that $\tZ(q)$ is bounded below for any $q\in\Z$.

\begin{defin} \label{def:genmotiviccohom}
The \emph{MW-motivic cohomology groups} $\HMW^{p,q}(X,\Z)$ are defined to be the hypercohomology groups 
\[
\HMW^{p,q}(X,\Z):=\mathbb H^p_{\mathrm{Zar}}(X,\tZ(q)).
\] 
These groups have a natural structure of $\KMW_0(k)$-modules.
\end{defin}

The groups $\HMW^{p,q}(X,\Z)$ are thus contravariant in $X \in \sm k$. By definition, we thus have $\HMW^{p,0}(X,\Z)=\H^{p}(X,\KMW_0)$ while the identification $(\sKMW_0)_{q}=\sKMW_q=\mathbf{W}$ for $q<0$ yields $\HMW^{p,q}(X,\Z)=\H^{p-q}(X,\mathbf{W})$ for $q<0$.

\begin{rem}
The presheaves $X\mapsto \HMW^{p,q}(X,\Z)$ are in fact presheaves with MW-transfers for any $p,q\in\Z$. This follows from \cite[Corollary 3.2.16]{Deglise16}.
\end{rem}

Recall next that we have an exact sequence of Zariski sheaves
\[
0\to \itZX X \to \tZX X\to \Ztr X\to 0
\]
for any smooth scheme $X$. If $X\to Y$ is a morphism of smooth schemes, we get a commutative diagram
\[
\xymatrix{0\ar[r] &  \itZX X \ar[r]\ar[d] & \tZX X\ar[r]\ar[d]  & \Ztr X\ar[r]\ar[d]  & 0 \\
0\ar[r] &  \itZX Y \ar[r] & \tZX Y\ar[r] & \Ztr Y\ar[r] & 0}
\]
with exact rows. It follows from \cite[Lemma 2.13]{Mazza06} that we get for any $q\geq 0$ an exact sequence of Zariski sheaves
\[
0\to \itZX {\gmpt^{\wedge q}} \to \tZX {\gmpt^{\wedge q}}\to \Ztr {\gmpt^{\wedge q}}\to 0.
\]
Now, the Suslin-Voevodsky construction is exact, and it follows that we have an exact sequence of complexes of Zariski sheaves
\[
0\to \Cstar(\itZX {\gmpt^{\wedge q}})[-q] \to \tZ(q)\to \Z(q)\to 0.
\]
We let the reader check that the complexes involved are complexes of $\KMW_0(k)$-modules.
\begin{defin}
For any $q\geq 0$, we set $\ItZ(q)=\Cstar(\itZX {\gmpt^{\wedge q}})[-q]$. We denote by 
\[
\HI^{p,q}(X,\Z)=\mathbb H^p_{\mathrm{Zar}}(X,\ItZ(q))
\]
the hypercohomology groups of the complex $\ItZ(q)$.
\end{defin}

The above exact sequence then yields a long exact sequence of motivic cohomology groups
\[
\ldots \to \HI^{p,q}(X,\Z)\to \HMW^{p,q}(X,\Z)\to \H^{p,q}(X,\Z)\to \HI^{p+1,q}(X,\Z)\to \ldots
\]
for any $q\geq 0$.

\subsection{Base change of the ground field}

The various constructions considered until now are functorial with respect to the ground field $k$, and we now give details about some of the aspects of this functoriality. Let $L$ be a field extension of $k$, both $L$ and $k$ assumed to be perfect. 
For any scheme $X$ over $k$, let $X_L=\spec L\times_{\spec k} X$ be its extension to $L$. When $X$ is smooth, the canonical bundle $\omega_X$ pulls-back over $X_L$ to the canonical bundle $\omega_{X_L}$ of $X_L$. 

By conservation of surjectivity and finiteness by base change, the pull-back induces a map $\Adm(X,Y) \to \Adm(X_L,Y_L)$. From the contravariant functoriality of Chow-Witt groups with support, passing to the limit, one then immediately obtains an \emph{extension of scalars} functor
\[
\xymatrix@1{
\cor k \ar[r]^-{\ext_{L/k}} & \cor L
}
\]
sending an object $X$ to $X_L$. This functor is monoidal, since $(X \times_k Y)_L \simeq X_L \times_L Y_L$ canonically.

When $L/k$ is finite, $L/k$ is automatically separable since $k$ is perfect (and $L$ is automatically perfect). Viewing an $L$-scheme $Y$ as a $k$-scheme $Y_{|k}$ by composing its structural morphism with the smooth morphism $\spec L \to \spec k$ defines a \emph{restriction of scalars} functor $\res_{L/K}: \sm L \to \sm k$. For any $L$-scheme $Y$ and $k$-scheme $X$, there are morphisms 
\[
\unit_Y:  Y \to L \times_k Y = (Y_{|k})_L \quad \text{and} \quad  \counit_X:(X_L)_{|k}= L \times_k X \to X
\]
induced by the structural morphism of $Y$ and $\id_Y$ for $\unit$, and the projection to $X$ for $\counit$. 

\begin{lem} \label{lem:finiterestriction}
Let $X \in \sm k$ and $Y,Y' \in \sm L$.
\begin{enumerate}
\item \label{item:adjsm} The morphisms of functors $\unit$ and $\counit$ given on objects by $\unit_Y$ and  $\counit_X$ define an adjunction $(\res_{L/k},\ext_{L/k})$. 
\item \label{item:proj} The natural morphism $(Y \times_L X_L)_{|k} \to Y_{|k} \times_k X$ is an isomorphism (adjoint to the morphism $\unit_Y \times_L \id_{X_L}$).
\item \label{item:unit} The morphism $\unit_Y$ is a closed embedding of codimension $0$, identifying $Y$ with the connected components of $(Y_{k})_L$ living over $\spec L$, diagonally inside $\spec L \times_k \spec L$.
\item \label{item:closedsm} The natural morphism $(Y \times_L Y')_{|k} \to Y_{|k} \times_k Y'_{|k}$ is a closed embedding of codimension $0$, thus identifying the source as a union of connected components of the target.
\end{enumerate}
\end{lem}

\begin{proof}
Part \eqref{item:adjsm} is straightforward. Part \eqref{item:proj} can be checked locally, when $X=\spec A$ with $A$ a $k$-algebra and $Y=\spec B$ with $B$ an $L$-algebra, and the morphism considered is the canonical isomorphism $B \otimes_L L \otimes_k A \simeq B \otimes_k A$. To prove \eqref{item:unit}, note that the general case follows from base-change to $Y$ of the case $Y=\spec L$, which corresponds to $L\otimes_k L \to L$ via multiplication. Since the source is a product of separable extensions of $L$, one of which is isomorphic to $L$ by the multiplication map, the claim holds.
Up to the isomorphism of part \eqref{item:proj}, the morphism $(Y \times_L Y')_{|k} \to Y_{|k} \times_k Y'_{|k}$ corresponds to the morphism $\unit_Y \times \id_{Y'}$ so \eqref{item:closedsm} follows from \eqref{item:unit} by base change to $Y'$. All schemes involved being reduced, the claims about identifications of irreducible components are clear.
\end{proof}

The functor $\res_{L/k}$ extends to a functor
$\xymatrix@1{
\cor L \ar[r]^-{\res_{L/k}} & \cor k.
}$
as we now explain. For any $X,Y \in \sm L$, we have $\Adm (X,Y)\subseteq \Adm (X_{|k},Y_{|k})$ by part \eqref{item:closedsm} of Lemma \ref{lem:finiterestriction}, so for any $T$ in it the push-forward induces a map 
\[
\chst dT{X \times Y}{\omega_{Y}} \to \chst dT{X_{|k} \times_k Y_{|k}}{\omega_{Y_{|k}}}
\]
because $\omega_{L/k}$ is canonically trivial. Passing to the limit, it gives a map $\cor L (X,Y) \to \cor k (X_{|k},Y_{|k})$, and it is compatible with the composition of correspondences; this is an exercise using base change, where the only nontrivial input is that when $X,Y,Z \in \sm L$, the push-forward from $X\times_L Y \times_L Z$ to $X_{|k} \times_k Y_{|k} \times_k Z_{|k}$ respects products, which follows from part \eqref{item:closedsm} of Lemma \ref{lem:finiterestriction}. 

The adjunction $(\res_{L/k},\ext_{L/k})$ between $\sm L$ and $\sm k$ from part \eqref{item:adjsm} of Lemma \ref{lem:finiterestriction} extends to an adjunction between $\cor L$ and $\cor k$, using the same unit and counit, to which we apply the graph functors to view them as correspondences (see after Definition \ref{def:cortilde}). 

We are now going to define another adjunction $(\ext_{L/k}, \res_{L/k})$ (note the reversed order) that only exists at the level of correspondences.
The unit and counit
\[
\tunit_X:  X \to (X_L)_{|k} \quad \text{and} \quad  \tcounit_Y:(Y_{|k})_L \to Y 
\]
are defined respectively as the transpose of $\counit$ and $\unit$, using Example \ref{ex:push-forwards}. By part \eqref{item:adjsm} of Lemma \ref{lem:finiterestriction} and by the composition properties of the transpose construction it is clear that they do define an adjunction $(\ext_{L/k},\res_{L/k})$. 

\begin{lem} \label{lem:compoextres}
The composition $\counit \circ \tunit$ is the multiplication (via the $\KMW_0(k)$-module structure) by the trace form of $L/k$ and the composition $\tcounit \circ \unit$ on $Y$ is the projection to the component of $(Y_{|k})_L$ corresponding to $Y$ and mentioned in \eqref{item:unit} of Lemma \ref{lem:finiterestriction}. 
\end{lem}
\begin{proof}
By Example \ref{ex:push-forwards}, we obtain that $\counit \circ \tunit$ is the multiplication by $\counit_* \langle 1\rangle$. Since $\counit_X$ is obtained by base change to $X$ of the structural map $\sigma:\spec L \to \spec k$, the element $\counit_* \langle 1 \rangle$ is actually the pull-back to $X$ of $\sigma_*\langle 1 \rangle$, which is the trace form of $L/k$ by the definition of finite push-forwards for Chow-Witt groups (or the Milnor-Witt sheaf $\sKMW_0$). Similarly, but this time by base change of the diagonal $\delta: \spec L \to \spec L \times \spec L$, we obtain that $\unit \circ \tcounit$ is the multiplication by $\delta_* \langle 1 \rangle$, which is the projector to the component corresponding to $\spec L$, and thus to $Y$ by base change. 
\end{proof}

The functors $\ext_{L/k}$ and $\res_{L/k}$ between $\sm k$ and $\sm L$ are trivially continuous for the Zariski topology: they send a covering to a covering and they preserve fiber products: $(X \times_Z X')_L \simeq X_L \times_{Z_L} X'_L$ and $(Y \times_T Y')_{|k}\simeq Y_{|k} \times_{T_{|k}} Y'_{|k}$. Therefore, they induce functors between categories of Zariski sheaves with transfers 
\[
\res_{L/k}^*: \sh k{\Zar}  \to \sh L{\Zar} \quad \text{and} \quad \ext_{L/k}^*: \sh L{\Zar} \to \Sh k{\Zar}.
\] 
In order to avoid confusion, we set $\bc_{L/k} =\res_{L/k}^*$ and $\tr_{L/k}=\ext_{L/k}^*$ in order to suggest the words ``base change'' and ``transfer'', but we still use the convenient notation $F_L$ for $\bc_{L/k}(F)$ and $G_{|k}$ for $\tr_{L/k}(G)$. We thus have, by definition
\[
F_L(U)=F(U_{|k})\qquad \text{and} \qquad G_{|k}(V)=G(V_L)
\]
for any $F \in \sh k{\Zar}$, $G \in \sh L{\Zar}$, $U \in \sm L$ and $V \in \sm k$.
It is formal that the adjunction $(\res_{L/k},\ext_{L/k})$ induces an adjunction $(\bc_{L/k},\tr_{L/k})$ with unit $\counit^*$ and counit $\unit^*$, while the adjunction $(\ext_{L/k},\res_{L/k})$ induces an adjunction $(\tr_{L/k},\bc_{L/k})$ with unit $\tcounit^*$ and counit $\tunit^*$. 

\begin{lem}
For any $X \in \sm k$ and $Y \in \sm L$, we have natural isomorphisms 
\[
(\tZX X)_L\simeq \tZX {X_L} \in \sh L{\Zar} \quad\text{and}\quad (\tZX Y)_{|k} \simeq \tZX {Y_{|k}} \in \sh k{\Zar}.
\]
In the same spirit, $(\uHom(\tZX X, \tZX k)_L\simeq \uHom(\tZX {X_L}, \tZX L) \in \sh L{\Zar}$, while on the other hand $\uHom(\tZX {Y}, \tZX L)_{|k} \simeq \uHom(\tZX {Y_{|k}},\tZX k) \in \sh k{\Zar}$.
\end{lem}

\begin{proof}
We have $(\tZX X)_L(U) \simeq \cor k (U_{|k},X) \simeq \cor L (U,X_L) \simeq \tZX {X_L}(U)$, by the adjunction $(\res_{L/k},\ext_{L/k})$ and $(\tZX Y)_{|k}(V)\simeq \cor L (V_L,Y) \simeq \cor k (V,Y_{|k}) \simeq \tZX {Y_{|k}}(V)$ by the transposed adjunction $(\ext_{L/k},\res_{L/k})$.

Similarly, $\uHom (\tZX X,\tZX k)_L(U)\simeq \cor k (U_{|k}\times X,k) \simeq \cor k ((U\times_L X_L)_{|k},k) \simeq \cor L (U\times_L X_L,L) \simeq \uHom (\tZX {X_L},\tZX L)(U)$. Finally, $\uHom (\tZX Y,\tZX L)_{|k}(V)\simeq \cor L (V_L \times_L Y,L) \simeq \cor k ((V_L \times_L Y)_{|k},k) \simeq \cor k (V\times Y_{|k},k) \simeq \uHom (\tZX {Y_{|k}},\tZX k)(V)$.
\end{proof}

\begin{coro} \label{coro:extressmash}
For any sequence of pointed schemes $(X_1,x_1),\ldots,(X_n,x_n)$, we have
\[
(\tZX {(X_1,x_1)\wedge \ldots \wedge (X_n,x_n)})_L \simeq \tZX {((X_1)_L,(x_1)_L)\wedge\ldots\wedge ((X_n)_L,(x_n)_L)} \in \sh L{\Zar}
\]
and
\[
(\uHom (\tZX {(X_1,x_1)\wedge \ldots \wedge (X_n,x_n)},\tZX k)_L \simeq \uHom (\tZX {\wedge_{i=1}^n((X_i)_L,(x_i)_L)},\tZX k)
\]
in $\sh L{\Zar}$.
\end{coro}
\begin{proof}
It immediately from the lemma applied to the split exact sequences defining the smash-products. 
\end{proof}

The same type of result for smash products would hold for restrictions, but the restriction of an $L$-pointed scheme is an $L_{|k}$-pointed scheme, not an $k$-pointed scheme. Nevertheless, the counit $\tunit^*$ of the adjunction $(\tr_{L/k},\bc_{L/k})$ induces maps
\begin{gather}
\big((\tZX {(X_1,x_1)\wedge \ldots \wedge (X_n,x_n)})_L\big)_{|k} \to \tZX {(X_1,x_1)\wedge\ldots\wedge (X_n,x_n)} \label{eq:mapsmashup} \\
\big(\uHom (\tZX {\wedge_{i=1}^n(X_i,x_i)},\tZX k)_L\big)_{|k} \to \uHom (\tZX {\wedge_{i=1}^n(X_i,x_i)},\tZX k). \label{eq:mapsmashdown}
\end{gather}
while the unit $\counit^*$ of the transposed adjunction $(\bc_{L/k},\tr_{L/k})$ induces maps
\begin{gather}
\tZX {(X_1,x_1)\wedge\ldots\wedge (X_n,x_n)} \to \big((\tZX {(X_1,x_1)\wedge \ldots \wedge (X_n,x_n)})_L\big)_{|k}  \label{eq:maptildesmashup} \\
\uHom (\tZX {\wedge_{i=1}^n(X_i,x_i)},\tZX k) \to \big((\uHom (\tZX {\wedge_{i=1}^n(X_i,x_i)},\tZX k))_L\big)_{|k}. \label{eq:maptildesmashdown}
\end{gather}

\begin{lem} \label{lem:extresCstar}
For any $F \in \sh k{\Zar}$ and any $G \in \sh L{\Zar}$, we have canonical isomorphisms $(\Cstar F)_L \simeq \Cstar (F_L)$ and $(\Cstar G)_{|k} \simeq \Cstar (G_{|k})$.
\end{lem}
\begin{proof}
It is straightforward, using $Y_{|k} \times \Delta^n \simeq (Y \times_L \Delta^n_L)_{|k}$ for $F$ and $X_L \times_L \Delta_L^n \simeq (X \times_k \Delta^n)_L$ for $G$.
\end{proof}

To avoid confusion, let us write $\tZ_k\{q\}$, $\tZ_k(q)$, etc. for the (complexes of) sheaves over $k$, and $\tZ_L\{q\}$, $\tZ_L(q)$ etc. for the same objects over $L$.

For any $q \in \Z$, using Corollary \ref{coro:extressmash}, we obtain, an isomorphism $\tZ_k\{q\}_L \simeq \tZ_L \{q\}$. Using the maps \eqref{eq:mapsmashup} for $q\geq 0$ and \eqref{eq:mapsmashdown} for $q<0$, applied to copies of $\gmpt$, we obtain a morphism $(\tZ_L\{q\})_{|k} \to \tZ\{q\}$. Symmetrically, using the maps \eqref{eq:maptildesmashup} and \eqref{eq:maptildesmashdown} we obtain a morphism $\tZ\{q\} \to (\tZ_L\{q\})_{|k}$. Using Lemma \ref{lem:extresCstar}, they induce an isomorphism and morphisms 
\begin{equation} \label{eq:transferZq}
\tZ_k(q)_L \simeq \tZ_L(q), 
\qquad
\xymatrix@C=3ex{(\tZ_L(q))_{|k} \ar[r]^-{\tunit^*} & \tZ(q)}
\qquad \text{and} \qquad 
\xymatrix@C=3ex{\tZ(q) \ar[r]^-{\counit^*} & (\tZ_L(q))_{|k}.}
\end{equation} 

\begin{lem} \label{lem:crossediso}
For any Zariski sheaf $F$ on $\sm L$ and any $X \in \sm L$, we have $\H^*(X,F_{|k})=\H^*(X_L,F)$. For any Zariski sheaf $G$ on $\sm k$ and any $Y \in \sm k$, we have $\H^*(Y,G_L)=\H^*(Y_{|k},G)$. More generally, if $F$ and $G$ are complexes of Zariski sheaves, we have the same result for hypercohomology.
\end{lem}
\begin{proof}
Let a Zariski sheaf $F$ be flabby if restricted to the small site of any scheme, it gives a flabby sheaf in the usual sense: restrictions are surjective. The Zariski cohomology can then be computed out of resolutions by such flabby sheaves. Both functors $\bc_{L/k}$ and $\tr_{L/k}$ preserve such flabby resolutions. So, given a flabby resolution $I^\bullet$ of $F$, we have 
\[
\H^i(X_L,F)=\H^i(I^\bullet(X_L))= \H^i(\ext_{L/k}^*I^\bullet (X))=\H^i(X,F_{|k}).
\]
A similar proof holds for $G$ and $Y$. The claim about hypercohomology is proved similarly using flabby Cartan-Eilenberg resolutions.
\end{proof}

Using the morphisms \eqref{eq:transferZq} and Lemma \ref{lem:crossediso}, we obtain for any $X \in \sm k$ and any $q$ two morphisms 
\[
\mathbb{H}^p(X, \tZ_k(q)) \to \mathbb{H}^p(X,(\tZ_L(q))_k) \simeq \mathbb{H}^p(X_L,\tZ_L(q))
\]
and
\[
\mathbb{H}^p(X_L, \tZ_L(q)) \simeq \mathbb{H}^p(X,(\tZ_L(q))_k) \to \mathbb{H}^p(X,\tZ_k(q)).
\]
in other words:
\begin{equation}
\xymatrix{
{\HMW^{p,q}}_k(X,\Z) \ar[r]^{\bc_{L/k}} & {\HMW^{p,q}}_L(X_L,\Z) 
} 
\qquad \text{and} \qquad 
\xymatrix{
{\HMW^{p,q}}_L(X_L,\Z) \ar[r]^{\tr_{L/k}} & {\HMW^{p,q}}_k(X,\Z).
}
\end{equation}
Using Lemma \ref{lem:compoextres}, we obtain 

\begin{lem}
On MW-motivic cohomology, the composition $\tr_{L/k} \circ \bc_{L/k}$ is the multiplication (via the $\KMW_0(k)$-module structure) by the trace form of the extension $L/k$.
\end{lem}

We now compare the MW-cohomology groups when computed over $k$ for a limit scheme that is of the form $X_L$ and when computed over some extension $L$ of $k$.

If $L$ is a finitely generated extension of $k$, that we view as the inverse limit $L=\varprojlim U$ of schemes $U \in \sm k$, we obtain a natural map
\[
\varinjlim_{T \in \Adm(U\times X,Y)} \hspace{-3ex}\chst dT {U \times X \times Y}{\omega_{Y}} \to \hspace{-1ex} \varinjlim_{T' \in \Adm(X_L,Y_L)}\hspace{-3ex}\chst d{T'} { X_L \times Y_L}{\omega_{Y_L}} 
\]
induced by pull-backs between the groups where $T'=T_L$.
When $d=\dim Y$, this is a map 
\[
\cor k (U \times X,Y) \to \cor L (X_L,Y_L)
\]
functorial in $U$. Taking the limit over $U$, we therefore obtain a map
\[
\xymatrix@1{
\cor k (X_L,Y) = \varinjlim \cor k(U \times X,Y) \ar[r]^-{\Psi_{L/k}} & \cor L (X_L,Y_L)
}
\]
using the construction of section \ref{sec:limits} to define the left hand side.
\begin{prop}
The map $\Psi_{L/k}$ is an isomorphism.
\end{prop}
\begin{proof}
To shorten the notation, we drop the canonical bundles in the whole proof, since they behave as they should by pull-back. 
The source of $\Psi_{L/k}$ can be defined using a single limit, as 
\[
\varinjlim_{(U,T)} \chs dT{U\times X \times Y}
\]
where the limit runs over the pairs $(U,T)$ with $(U,T)\leq (U',T')$ if there is a map $U' \subseteq U$ and $T\cap U' \subseteq T'$. The corresponding transition map on Chow-Witt groups is the restriction to $U'$ composed with the extension of support from $T \cap U'$ to $T'$. Note that both of these maps are injective, as explained in the proof of Lemma \ref{lem:restrictioninj} for the first one and right before that same Lemma for the second one. 
The maps $f_{U,T}: \chs dT{U \times X \times Y} \to \varinjlim_{V \subseteq U} \chs d{T \cap V}{V \times X \times Y}$ sending the initial group in the direct system to its limit are again injective, and their target can be identified with $\chs d{T_L}{X_L \times Y_L}$ by Lemma \ref{lem:limitChowiso} (independently of $U$, except that $T$ lives on $U$). The $f_{U,T}$ therefore induce an injective map
\begin{equation} \label{eq:limitmap}
\varinjlim_{(U,T)} \chs dT{U\times X \times Y} \to \varinjlim_{(U,T)} \chs d{T_L}{X_L \times Y_L}
\end{equation}
This map is also surjective because any $\chs d{T\cap V}{V \times X \times Y}$ in the target of $f_{U,T}$ is surjected by the same group on the source, after passing from $(U,T)$ to $(V, T \cap V)$. Finally, the isomorphism \eqref{eq:limitmap} actually maps to $\cor L(X_L,Y_L)$ because any $T' \in \Adm(X_L,Y_L)$ actually comes by pull-back from some open $U$, by \cite[8.3.11]{EGAIV3}. It is easy to see that in defining that isomorphism we have just expanded the definition of $\Psi_{L/k}$
\end{proof}


\bibliography{FGW}{}
\bibliographystyle{plain}


\end{document}